\documentclass[brochure,12pt]{bourbaki}
\usepackage[matrix,arrow]{xy}

\usepackage{footnote,mathrsfs,mathabx}

\usepackage[T1]{fontenc}
\usepackage{lmodern}
\usepackage[french]{babel}
\usepackage[fixlanguage]{babelbib}
\selectbiblanguage{french}

\newcommand{\Hom}{\operatorname{Hom}}
\newcommand{\Stab}{\operatorname{Stab}}
\newcommand{\Ker}{\operatorname{Ker}}

\newcommand{\ch}{\operatorname{ch}}
\newcommand{\td}{\operatorname{td}}

\newcommand{\Amp}{\operatorname{Amp}}
\newcommand{\BAmp}{\operatorname{BAmp}}
\newcommand{\BIG}{\operatorname{Big}}
\newcommand{\Mov}{\operatorname{Mov}}
\newcommand{\NS}{\operatorname{NS}}

\def\Z{\mathbb{Z}}

\def\Q{\mathbb{Q}}
\def\R{\mathbb{R}}
\def\C{\mathbb{C}}

\def\G{\mathbb{G}}

\newcommand{\cA}{{\mathcal  A}}
\newcommand{\cB}{{\mathcal  B}}

\newcommand{\cD}{{\mathcal  D}}
\newcommand{\cE}{{\mathcal  E}}
\newcommand{\cF}{{\mathcal  F}}

\newcommand{\cO}{{\mathcal O}}

\newcommand{\cP}{{\mathcal P}}
\newcommand{\cQ}{{\mathcal Q}}
\newcommand{\cT}{{\mathcal T}}

\newcommand{\cW}{{\mathcal W}}

\newcommand{\ra}{\rightarrow}

\date{Juin 2018}
\bbkannee{70\textsuperscript{e} ann\'ee, 2017--2018}
\bbknumero{1147}
\title{Conditions de stabilité et géométrie birationnelle}
\subtitle{d'après Bridgeland, Bayer-Macrì,...}
\author{Fran\c cois CHARLES}
\address{Laboratoire de mathématiques d'Orsay\\ UMR 8628 du CNRS\\ Universit{é}
Paris-Sud,
B{â}timent 425
\\91405 Orsay cedex, France}
\email{francois.charles@math.u-psud.fr}

\begin{document}

\maketitle

\noindent{\bf INTRODUCTION}

\bigskip

L'objet principal de ce texte est la notion de condition de stabilité sur une catégorie triangulée $\cD$. Inspiré par des travaux de physique théorique, Bridgeland a introduit cette notion dans \cite{Bridgeland07}. Un des aspects particulièrement frappants de celle-ci est le fait que l'ensemble des conditions de stabilité sur $\cD$ --- supposées se factoriser par un groupe de type fini --- est une variété différentielle $\Stab(\cD)$ de dimension finie, étale sur un espace vectoriel complexe. L'espace $\cD$ est muni de structures de décomposition en chambres et murs associées à certains paramètres.

Un cas particulier important est celui où la catégorie $\cD$ est la catégorie dérivée $D^b(X)$ de la catégorie des faisceaux cohérents sur une variété $X$ projective lisse sur $\C$. Dans ce cas, la notion de condition de stabilité étend, au moins si $X$ est une surface, celle de condition de stabilité à la Gieseker, et la décomposition en chambres et murs étend de telles décompositions venant de la théorie des variations de quotient GIT.

Dans le cas où $X$ est une surface $\mathrm{K}3$, Bayer-Macr\`i dans \cite{BayerMacri14proj, BayerMacri14} relient la décomposition d'une composante de $\Stab(X)$ à la décomposition de certains cônes de diviseurs dans des espaces de modules $M$ de faisceaux sur $X$, qui sont des variétés symplectiques holomorphes. Ils montrent, en un sens précis, que le programme du modèle minimal sur un tel $M$ est induit par des phénomènes de wall-crossing sur $\Stab(X)$. Le but de ce travail est d'exposer ces développements.

\bigskip

Nous avons dû laisser de côté plusieurs applications importantes des conditions de stabilité à la géométrie algébrique. Nous ne mentionnerons donc pas le lien avec la symétrie miroir ou la géométrie énumérative. Nous n'évoquerons pas non plus les récents développements autour de la construction de conditions de stabilité en dimension $3$ et des inégalités de Bogomolov en dimension supérieure. Nous renvoyons aux beaux textes de survol \cite{Bridgeland09, Huybrechts14, MacriSchmidt17, Huizenga17} pour des compléments sur cet exposé.

\bigskip

Les deux premières sections de ce texte sont consacrées à un exposé des résultats de Bridgeland \cite{Bridgeland07}, revus par Kontsevich-Soibelman \cite{KontsevichSoibelman08} sur les conditions de stabilité sur une catégorie triangulée arbitraire. Nous adaptons des simplifications de Bayer \cite{Bayer16}. Dans la section 3, nous donnons des exemples de condition de stabilité sur les courbes et les surfaces, d'après notamment \cite{ArcaraBertram13}, et décrivons d'après \cite{Bridgeland08} une composante de l'espace des conditions de stabilité sur une surface $\mathrm{K}3$, simplifiant là encore la preuve en adaptant \cite{Bayer16}. Les sections 4 et 5 sont consacrées aux travaux de Bayer-Macr\`i \cite{BayerMacri14proj, BayerMacri14}.

\bigskip

\noindent{\bf Remerciements}. Je remercie chaleureusement Arend Bayer, Tom Bridgeland et Daniel Huybrechts pour leurs commentaires précieux sur une première version de ce texte, ainsi qu'Antoine Chambert-Loir pour ses nombreuses remarques, et Viviane Le Dret pour sa relecture attentive.

\section{Conditions de stabilité dans une catégorie triangulée}

Dans cette section, nous introduisons la définition, due à Bridgeland \cite{Bridgeland07}, de condition de stabilité sur une catégorie triangulée arbitraire. La question de la construction de conditions de stabilité sera abordée plus bas dans le texte, nous nous contentons ici de considérations abstraites.

\subsection{Stabilité sur une catégorie abélienne}\label{subsection:stab-ab}

Fixons dans ce qui suit une catégorie abélienne $\cA$. On notera comme d'habitude $K_0(\cA)$ le groupe de Grothendieck de $\cA$.

\begin{defi}\label{definition:stability-function}
Une fonction de stabilité sur $\cA$ est un morphisme de groupes additifs 
$$Z : K_0(\cA)\ra \C$$
tel que pour tout objet non nul $A$ de $\cA$, $Z(A)$ appartienne à la réunion du demi-plan de Poincaré et de la demi-droite réelle strictement négative $\R^*_{-}$.
\end{defi}

Étant donnée une telle fonction de stabilité $Z$, on notera souvent 
$$Z(X)=-d(X)+i\,r(X),$$
où $d(X)$ et $r(X)$ sont des nombres réels, appelés respectivement le \emph{degré} et le \emph{rang généralisés} de l'élément $X$ de $K_0(\cA)$. La \emph{pente généralisée} de $X$ est le réel $\frac{d(X)}{r(X)}$, défini si $r(X)$ est non nul.

La pente généralisée n'étant pas toujours bien définie, il est plus agréable de travailler avec la phase. Cela nous permet de définir les objets stables et semistables.

\begin{defi}
Soit $Z$ une fonction de stabilité sur $\cA$. Si $X$ est un élément non nul de $K_0(\cA)$, la phase de $X$, notée $\phi(X)$, est l'unique réel $\phi\in ]0, 1]$ tel que 
$$Z(X)\in \R^*_+e^{i\pi\phi}.$$
\end{defi}

\begin{defi}\label{definition:ss}
Soit $Z$ une fonction de stabilité sur $\cA$, et soit $A$ un objet non nul de $\cA$. On dit que $A$ est stable (resp. semistable) si pour tout sous-objet strict non nul~$B$ de~$A$, on a 
$$\phi(B)<\phi(A)$$
(resp. $\phi(B)\leq \phi(A)$).
\end{defi}

Si $0\ra B\ra A\ra C\ra 0$ est une suite exacte d'objets non nuls de $\cA$, on a $Z(A)=Z(B)+Z(C)$. Les trois nombres complexes $Z(A), Z(B)$ et $Z(C)$ étant tous trois des éléments de $\mathbb H\cup \R^*_{-}$, on a 
$$\phi(B)\leq \phi(A)\iff \phi(C)\geq \phi(A)$$
et
$$\phi(B)\geq \phi(A)\iff \phi(C)\leq \phi(A).$$
En particulier, un objet $A$ de $\cA$ est (semi)stable si et seulement si pour tout quotient strict  non nul $C$ de $A$, on a $\phi(C)>\phi(A)$ (resp. $\phi(C)\geq\phi(A)$). 

Les définitions ci-dessus sont habituelles si $\cA$ est la catégorie des faisceaux cohérents sur une courbe intègre projective et lisse sur un corps. Dans ce cas, les fonctions $d$ et $r$ qui à un faisceau cohérent associent respectivement son degré et son rang au point générique s'étendent en des morphismes de groupes $d, r : K_0(\cA)\ra \R$, d'où une fonction de stabilité $Z=-d+i\,r$. La définition \ref{definition:ss} étend aux faisceaux cohérents arbitraires la notion de (semi)stabilité exprimée d'habitude en fonction des pentes \cite{HuybrechtsLehn10}, de telle sorte que les faisceaux de torsion sont tous semistables de pente~$1$.

Les notions de pente, phase, stabilité, etc. définies ci-dessus dépendent de la fonction de stabilité $Z$ --- autant que possible, nous garderons implicite cette dépendance.

\bigskip

Nous pouvons maintenant définir la notion de précondition de stabilité.

\begin{defi}\label{definition:condition-stab-ab}
Soit $\cA$ une catégorie abélienne, et soit $Z : K_0(\cA)\ra\C$ une fonction de stabilité. On dit que le couple $(\cA, Z)$ est une précondition de stabilité abélienne si tout objet non nul $A$ de $\cA$ admet une filtration finie
$$0=A_0\subset A_1\subset\cdots\subset A_n=A$$
telle que
\begin{itemize}
\item pour tout $i\geq 0$, $A_{i+1}/A_i$ est semistable;
\item pour tout $i\geq 0$, $\phi(A_{i+1}/A_i)>\phi(A_{i+2}/A_{i+1})$.
\end{itemize}
\end{defi}

L'unicité de la filtration qui apparaît dans la définition ci-dessus est garantie par le lemme élémentaire suivant --- nous appellerons donc cette filtration la \emph{filtration de Harder-Narasimhan} de $A$.

\begin{lemm}\label{lemme:ss-nul}
Soit $Z$ une fonction de stabilité sur $\cA$, et soient $A$ et $B$ deux objets semistables de $\cA$. Si $\phi(A)>\phi(B)$, alors il n'existe pas de morphisme non nul de~$A$ vers~$B$.
\end{lemm}

\begin{proof}
Soient $A$ et $B$ comme dans l'énoncé, et soit $f$ un morphisme non nul de $A$ vers $B$. Soit $C$ l'image de $f$. Puisque $C$ est un quotient de $A$, qui est semistable, on a $\phi(A)\leq\phi(C)$. Comme $C$ est un sous-objet de $B$, on a de même $\phi(C)\leq \phi(B)$, d'où une contradiction.
\end{proof}

\subsection{Stabilité dans une catégorie triangulée}

Dans la suite de ce texte, on étudiera les familles de conditions de stabilité. Un des points clés de la théorie est le fait que, dans la plupart des cas, on souhaitera non seulement déformer la fonction de stabilité $Z$, mais aussi la catégorie abélienne $\cA$ à l'intérieur de la catégorie triangulée $D^b(\cA)$. Pour ce faire, il est plus commode de travailler avec des préconditions de stabilité sur une catégorie triangulée.

\bigskip

Soit $(\cA, Z)$ une condition de stabilité sur une catégorie abélienne, et considérons la catégorie triangulée $\cD=D^b(\cA)$. Si $\phi$ est un réel dans $\mathopen]0, 1]$, soit $\cP(\phi)$ la sous-catégorie pleine de $\cA$ dont les objets sont les objets de $\cA$ semistables de phase $\phi$. Si $\phi$ est un réel arbitraire de la forme
$$\phi=n+\phi_0,$$
avec $n\in\Z$ et $0<\phi_0\leq 1,$ soit $\cP(\phi)$ la sous-catégorie $\cP(\phi_0)[n]$ de $\cD$. Les $\cP(\phi)$ sont des catégories additives. Les propriétés qu'elles vérifient justifient la définition suivante.

\begin{defi}
Soit $\cD$ une catégorie triangulée. Un découpage de $\cD$ est la donnée, pour tout nombre réel $\phi$, d'une sous-catégorie additive pleine $\cP(\phi)$ de $\cD$, de telle façon que les conditions suivantes soient vérifiées : 
\begin{enumerate}
\item Pour tout $\phi\in\R$, $\cP(\phi+1)=\cP(\phi)[1]$;
\item Soient $\phi_1$ et $\phi_2$ deux réels, et soit $A_i$ un objet de $\cP(\phi_i)$, $i=1, 2$. Si $\phi_1>\phi_2$, alors $\Hom_{\cD}(A_1, A_2)=0$;
\item Soit $A$ un objet de $\cD$. Il existe des réels $\phi_1>\cdots>\phi_n$, des objets \mbox{$A_0=0, \ldots, A_n=A$} de $\cD$, et des triangles distingués
$$A_i\ra A_{i+1}\ra E_i\ra A_i[1],$$
où $E_i$ est un objet non nul de $\cP(\phi_i)$. 
\end{enumerate}
\end{defi}

\begin{rema}\label{remarque:strictly-full}
Si $A$ est un objet de $\cD$ isomorphe à un objet de $\cP(\phi)$, alors $A$ est un objet de $\cP(\phi)$ : les catégories $\cP(\phi)$ sont strictement pleines --- utiliser la condition 3 pour $A$ et le lemme \ref{lemme:orthogonalite} ci-dessous.
\end{rema}

Nous laissons au lecteur le soin de vérifier que si $(\cA, Z)$ est une précondition de stabilité sur une catégorie abélienne, alors la construction ci-dessus munit $D^b(\cA)$ d'un découpage naturel. La première propriété est satisfaite par construction, la seconde grâce au lemme \ref{lemme:ss-nul}, et la troisième grâce à l'existence des filtrations de Harder-Narasimhan.

Soit $\cP$ un découpage d'une catégorie triangulée $\cD$. Comme dans le cas des filtrations de Harder-Narasimhan, on vérifie que la condition 2 ci-dessus garantit l'unicité à isomorphisme près des triangles apparaissant dans 3. En particulier, les phases $\phi_0, \ldots, \phi_{n-1}$ associées à un objet $A$ de $\cD$ comme dans 3 sont bien définies. On notera 
$$\phi^+(A)=\phi_0, \phi^-(A)=\phi_{n-1}.$$
Remarquons que l'on a $\phi^+(A)=\phi^-(A)$ si et seulement si $A$ est un objet de $\cP(\phi)$ pour $\phi=\phi^+(A)=\phi^-(A).$

%\begin{prop}
%Soit $A$ un objet de $\cD$. Alors
%$$\phi^+(A)=\max\{\phi\in\R, \exists B\in\mathrm{Ob}(\cP(\phi)), \Hom_\cD(B, A)\neq 0\}$$
%et
%$$\phi^-(A)=\min\{\phi\in\R, \exists B\in\mathrm{Ob}(\cP(\phi)), \Hom_\cD(A, B)\neq 0\}.$$
%\end{prop}
%

\bigskip

La notion de découpage est étroitement reliée à celle de t-structure, expliquons en quel sens. Si $I$ est un intervalle de $\R$, soit $\cP(I)$ la sous-catégorie pleine de $\cD$ dont les objets sont les $A$ tels que $\phi^-(A)$ et $\phi^+(A)$ sont dans $I$. Autrement dit, $\cP(I)$ est la plus petite sous-catégorie pleine de $\cD$ contenant les $\cP(\phi)$, $\phi\in I$ et stable par extensions. On note $\cP(\geq \phi)$ (resp. $\cP(\leq \phi), \cP(> \phi), \cP(< \phi)$) pour la catégorie $\cP([\phi, \infty\mathclose[)$ (resp. $\cP(\mathopen]-\infty, \phi]), \cP(\mathopen]\phi, \infty\mathclose[), \cP(\mathopen]-\infty, \phi\mathclose[)$).

La condition 2 de la définition a pour conséquence immédiate la propriété suivante d'orthogonalité pour les catégories $\cP(I)$. 

\begin{lemm}\label{lemme:orthogonalite}
Soient $A_1$ et $A_2$ deux objets de $\cD$. Si $\phi^-(A_1)>\phi^+(A_2)$, alors 
$$\Hom_\cD(A_1, A_2)=0.$$
En particulier, si $I$ et $J$ sont deux intervalles disjoints de $\R$, avec $i>j$  pour tous $i\in I, j\in J$, alors pour tous $A_1\in \cP(I), A_2\in \cP(J)$, 
$$\Hom_\cD(A_1, A_2)=0.$$
\end{lemm}

\begin{prop}\label{proposition:t-structure}
Soit $\cD$ une catégorie triangulée munie d'un découpage $\cP$. Si $\phi$ est un réel arbitraire, les sous-catégories de $\cD$
$$\cD^{\geq 0}:=\cP(\leq \phi+1), \cD^{\leq 0}:=\cP(> \phi)$$
définissent une t-structure bornée sur $\cD$, de c\oe{}ur $\cP(]\phi, \phi+1])$. 
\end{prop}

\begin{proof}
Si $n$ est un entier, on a 
$$\cD^{\geq n}=\cD^{\geq 0}[-n]=\cP(\leq \phi+1)[-n]=\cP(\leq \phi-n+1)$$
et
$$\cD^{\leq n}=\cD^{\leq 0}[-n]=\cP(> \phi)[-n]=\cP(>\phi-n).$$
En particulier, on a bien $\cD^{\leq -1}\subset \cD^{\leq 0}$. 

Le lemme ci-dessus garantit que pour $A_1\in \cD^{\leq 0}=\cP(> \phi)$, $A_2\in \cD^{\geq 1}=\cP(\leq \phi)$, on a $\Hom_\cD(A_1, A_2)=0.$

Enfin, si $A$ est un objet arbitraire de $\cD$, considérons des objets $A_0=0, \ldots, A_n=A$ de $\cD$, et des triangles distingués
$$A_i\ra A_{i+1}\ra E_i\ra A_i[1],$$
où $E_i$ est un objet de $\cP(\phi_i)$, avec $\phi_0>\cdots>\phi_{n-1}$. Soit $i$ le plus grand entier entre $0$ et $n$ tel que $\phi_j>\phi$ pour tout $j<i$. Alors $A_i$ est extension successive d'objets $E_0, \ldots, E_{i-1}$ de $\cP(>\phi)=\cD^{\leq 0}$, et l'on a un triangle distingué
$$A_i\ra A\ra E\ra A[1],$$
où $E$ est extension successive d'objets $E_i, \ldots, E_{n-1}$ de $\cP(\leq \phi)=\cD^{\geq 1}$.

On vient de montrer que $\cD^{\geq 0}=\cP(\leq \phi+1), \cD^{\leq 0}=\cP(> \phi)$ définit bien une t-structure sur $\cD$. On a de plus 
$$\bigcup_{m, n\in \Z} \cD^{\geq n}\cap \cD^{\leq m}=\bigcup_{m, n\in \Z} \cP(>\phi-m)\cap\cP(\leq \phi-n+1)=\bigcup_{m, n\geq 0} \cP(]\phi-m, \phi+n])=\cD,$$
ce qui montre que la t-structure est bornée.
\end{proof}

Un découpage d'une catégorie triangulée $\cD$ fournit donc une famille compatible de t-structures bornées sur $\cD$. Par convention, nous dirons que le \emph{c\oe{}ur} du découpage $\cP$ est la catégorie abélienne
$$\cP(\mathopen]0, 1]).$$
C'est bien le c\oe{}ur d'une t-structure bornée sur~$\cD$.

\bigskip

Bien entendu, le groupe des bijections croissantes $f : \R\ra\R$ telles que $f(\phi+1)=f(\phi)+1$ pour tout $\phi\in\R$ agit sur l'ensemble des découpages de~$\cD$ via $\cP(\phi)\mapsto \cP(f(\phi))$. On rigidifie la situation en introduisant l'analogue triangulé de la fonction de stabilité.

\begin{defi}\label{definition:condition-stabilite}
Soit $\cD$ une catégorie triangulée. Une précondition de stabilité sur~$\cD$ est un couple $\sigma=(\cP, Z)$, où $\cP$ est un découpage de $\cD$, et $Z$ est un morphisme de groupes additifs
$$Z : K_0(\cD)\ra\C$$
tel que pour tout réel $\phi$ et tout objet non nul $A$ de $\cP(\phi)$, on ait
$$Z(A)\in \R^*_+e^{i\pi\phi}.$$
On dit que la fonction $Z$ est la fonction charge centrale.
\end{defi}

Si $\phi$ est un réel quelconque, les objets non nuls de $\cP(\phi)$ sont les objets \emph{($\sigma$-)semistables de phase $\phi$}. Les objets simples sont les objets \emph{($\sigma$-)stables de phase $\phi$}.
\bigskip

La discussion précédente montre l'équivalence de nos deux définitions de précondition de stabilité, au sens suivant. Nous en laissons les détails au lecteur \cite[Proposition~5.3]{Bridgeland07}.

\begin{prop}
Soit $\cD$ une catégorie triangulée. La donnée d'une précondition de stabilité sur $\cD$ est équivalente à la donnée d'une catégorie abélienne $\cA$, c\oe{}ur d'une t-structure bornée sur $\cD$, et d'une fonction de stabilité $K_0(\cA)\ra \C$ telles que $(\cA, Z)$ est une précondition de stabilité abélienne.
\end{prop}

Dans ce qui suit, nous utiliserons sans la mentionner l'équivalence des définitions précédentes. Une précondition de stabilité $\sigma$ sur $\cD$ sera donc considérée comme un couple $(\cA, Z)$ ou $(\cP, Z)$ selon la situation. Les notations rendront clair l'objet considéré.

\begin{prop}\label{proposition:cP-ab}
Soit $(\cP, Z)$ une précondition de stabilité sur une catégorie triangulée $\cD$, et soit $\phi\in\R$. La catégorie $\cP(\phi)$ est abélienne.
\end{prop}

\begin{proof}
Quitte à décaler, on peut supposer $\phi\in]0, 1]$. En particulier, $\cP(\phi)$ est une sous-catégorie pleine du c\oe{}ur $\cA=\cP(]0, 1])$ de $\cP$. Pour montrer le résultat, il suffit de montrer que si $f : A\ra B$ est un morphisme entre objets de $\cP(\phi)$, l'image et le noyau de $f$ dans $\cA$ sont des objets de $\cP(\phi)$. Cela garantit que ce sont aussi l'image et le noyau de $f$ dans $\cP(\phi)$.

Soit $I$ l'image de $f$. Comme $I$ est un quotient de $A$ qui est semistable de phase $\phi$, et comme $I$ admet par définition un quotient semistable de phase $\phi^-(I)$, on a $\phi=\phi(A)\leq \phi^-(I)$. De même, on a $\phi^+(I)\leq \phi(B)=\phi.$ Ainsi, $\phi^-(I)=\phi^+(I)=\phi$, ce qui montre que $I$ est un objet de $\cP(\phi)$. 

Soit $K$ le noyau de $f$. On a une suite exacte
$$0\ra K\ra A\ra I\ra 0.$$
D'après ce qui précède, $I$ est un objet de $\cP(\phi)$. De $Z(A)=Z(K)+Z(I)$ avec $Z(A), Z(I)\in \R^*_+e^{i\pi\phi}$, on tire $Z(K)\in \R^*_+e^{i\pi\phi}$. Par ailleurs, comme précédemment, on a $\phi^+(K)\leq \phi$, donc les phases $\phi_i$ de $K$ sont toutes inférieures ou égales à $\phi$. Par définition, $K$ est un objet de $\cA$, donc ces phases sont toutes strictement positives. La filtration de Harder-Narasimhan de $K$ permet ainsi d'écrire $Z(K)$ comme somme d'éléments de $\R^*_+e^{i\pi\phi_i}$ avec $0<\phi_i\leq \phi$. Cela n'est possible que si $K$ n'a qu'une seule phase, égale à~$\phi$, soit si $K$ est semistable de phase $\phi$.
\end{proof}

%\begin{rema}
%Les arguments ci-dessus permettent de montrer que si $I$ est un intervalle de $\R$ strictement inclus dans un intervalle de taille $1$, alors la catégorie $\cP(I)$ est abélienne, voir \cite{Bridgeland}.
%\end{rema}

\subsection{Actions de groupes}\label{subsection:action}

Soit $\widetilde{GL}_2^+(\R)$ le revêtement universel du groupe $GL_2^+(\R)$, composante neutre de $GL_2(\R)$. Ce groupe s'identifie au groupe des couples $(a, g)$, où $g\in GL_2^+(\R)$ et $a$ est une fonction continue, nécessairement croissante, de $\R$ dans $\R$, telle que pour $\phi\in \R$, $a(\phi+1)=a(\phi)+1$ et 
$$g(\cos(\pi\phi), \sin(\pi\phi))\in \R^*_+(\cos(\pi a(\phi)), \sin(\pi a(\phi))).$$

Le groupe $\widetilde{GL}_2^+(\R)$ agit à droite sur l'ensemble des préconditions de stabilité sur une catégorie triangulée $\cD$ par la formule 
$$(\cP, Z).(a, g)=(\cP\circ a, g^{-1}\circ Z),$$
où $g$ agit sur $\C$ via l'identification naturelle de l'espace vectoriel $\C$ à $\R^2$ et où $\cP\circ a$ est le découpage de $\cD$ tel que pour tout $\phi\in\R$,
$$(\cP\circ a)(\phi)=\cP(a(\phi)).$$
Remarquons que l'action de $\widetilde{GL}_2^+(\R)$ préserve l'ensemble des objets semistables de $\cD$, ainsi que le noyau de la charge centrale.

\bigskip

Soit $\mathrm{Aut}(\cD)$ le groupe des autoéquivalences exactes de $\cD$. Le groupe $\mathrm{Aut}(\cD)$ agit à gauche sur l'ensemble des préconditions de stabilité sur $\cD$ par la formule
$$T(\cP, Z)=(T.\cP, Z\circ T^{-1}).$$
Cette action commute manifestement à l'action de $\widetilde{GL}_2^+(\R)$.

\subsection{La condition de support}

Soit $\cD$ une catégorie triangulée. Dans la suite de ce texte, quand il sera question de (pré)conditions de stabilité sur $\cD$, on aura toujours fixé a priori un groupe abélien libre de type fini $\Lambda$ et un morphisme de groupes abéliens 
$$v : K_0(\cD)\ra \Lambda.$$
On imposera toujours aux charges centrales $Z : K_0(\cD)\ra \C$ de se factoriser par $v$, via un morphisme lui aussi noté $Z : \Lambda\ra \C$. On gardera souvent $v$ et $\Lambda$ implicites.

Dans ce contexte, le groupe $\mathrm{Aut}_\Lambda(\cD)$ constitué des couples $(T, \phi)$ formés d'une autoéquivalence exacte $T$ de $\cD$ et d'un automorphisme $\phi$ de $\Lambda$ tels que $v\circ T=\phi\circ v$ agit comme plus haut sur l'ensemble des préconditions de stabilité sur $\cD$ qui se factorisent par $v$.

\bigskip

La notion de condition de stabilité s'obtient en imposant une propriété de finitude aux préconditions de stabilité. Nous utiliserons l'approche de Kontsevich-Soibelman \cite{KontsevichSoibelman08}, plus restrictive que celle originale de Bridgeland mais suffisante dans le cas qui nous occupe. Elle repose sur la remarque suivante.

\begin{prop}\label{proposition:equivalent-support}
Soit $\cD$ une catégorie triangulée, et soit $\sigma=(\cP, Z)$ une précondition de stabilité sur $\cD$. Supposons que $Z$ se factorise par $v : K_0(\cD)\ra \Lambda$, où $\Lambda$ est un groupe abélien libre de type fini. Les conditions suivantes sont équivalentes : 
\begin{enumerate}
\item Pour toute norme $\Vert .\Vert $ sur $\Lambda_\R$, il existe une constante $C>0$ telle que pour tout objet $\sigma$-semistable $A$ de $\cD$, on ait 
$$|Z(A)|\geq C\Vert v(A)\Vert \ ;$$
\item Il existe une forme quadratique sur $\Lambda_\R$, définie négative sur le noyau de $Z$, telle que si $A$ est un objet $\sigma$-semistable de $\cD$, alors $Q(v(A))\geq 0$.
\end{enumerate}
\end{prop}

\begin{proof}
Supposons la première condition satisfaite. Soit $\Vert .\Vert $ une norme euclidienne sur $\Lambda_\R$, et $C$ la constante correspondante. Soit $Q$ la forme quadratique sur $\Lambda_\R$ telle que, pour tout $w\in\Lambda$, 
$$Q(w)=|Z(w)|^2-C^2\Vert w\Vert ^2.$$
Manifestement, $Q$ vérifie les propriétés de la seconde condition. 

Réciproquement, supposons l'existence de $Q$ comme dans $2$. Soit $\Lambda'$ l'orthogonal du noyau de $Z$ dans $\Lambda_\R$. Alors la restriction de $Z$ à $\Lambda'$ est injective, et la forme quadratique 
$$a+b\mapsto \Vert a+b\Vert ^2:=-Q(a)+|Z(b)|^2,$$
où $a\in \Ker Z$ et $b\in\Lambda'$ définit une norme $\Vert .\Vert $ sur $\Lambda_\R$. Il suffit de vérifier la condition $1$ pour cette norme. 

Soit $K>0$ une constante telle que pour tout $b$ dans $\Lambda'$, on ait $KQ(b)\leq |Z(b)|^2.$ Soit $A$ un objet $\sigma$-semistable de $\cD$. On écrit $v(A)=a+b, a\in \Ker Z, b\in\Lambda'.$ Alors $KQ(v(A))=KQ(a)+KQ(b)\geq 0,$ d'où $KQ(a)+|Z(b)|^2=KQ(a)+|Z(A)|^2\geq 0$. Ainsi, on a 
$$|Z(A)|^2\geq -KQ(a)=K\Vert v(A)\Vert ^2-K|Z(A)|^2,$$
i.e.
$$|Z(A)|^2\geq \frac{K}{1+K} \Vert v(A)\Vert ^2,$$
ce qui conclut.
\end{proof}

\begin{rema}\label{remarque:uniformite}
La preuve montre en fait le résultat suivant : étant donnée une forme quadratique $Q$ sur $\Lambda_\R$,  il existe une constante $C>0$ telle que si $\sigma$ est une condition de stabilité satisfaisant la condition du second énoncé, le premier énoncé vaut pour $\sigma$ avec la constante $C$.
\end{rema}

\begin{rema}
Il suffit de vérifier les deux conditions de la proposition pour les objets $\sigma$-semistables de phase dans $\mathopen]0,1]$, i.e., pour les objets semistables dans le c\oe{}ur de $\cP$.
\end{rema}

La définition suivante est fondamentale.

\begin{defi}\label{definition:cs}
Soit $\cD$ une catégorie triangulée, et soit $\sigma=(\cP, Z)$ une précondition de stabilité sur $\cD$. On dit que $\sigma$ satisfait la condition de support si les conditions équivalentes de la proposition précédente sont satisfaites. Une condition de stabilité sur~$\cD$ est une précondition de stabilité qui satisfait la condition de support. 
\end{defi}

On dira aussi que $\sigma$ satisfait la condition de support par rapport à la forme quadratique $Q$ sur $\Lambda_\R$ si $Q$ satisfait la seconde condition de la proposition précédente.

\begin{rema}\label{remarque:support-action}
Les actions naturelles de $\widetilde{GL}_2^+(\R)$ et $\mathrm{Aut}_\Lambda(\cD)$ sur l'ensemble des préconditions de stabilité induisent des actions sur l'ensemble des conditions de stabilité. Plus précisément, si $\sigma$ satisfait la condition de support par rapport à $Q$, alors la discussion de \ref{subsection:action} montre que toute condition de stabilité dans la $\widetilde{GL}_2^+(\R)$-orbite de $\sigma$ satisfait la condition de support par rapport à $Q$.
\end{rema}

On peut tester la condition de support sur les objets stables.

\begin{prop}\label{proposition:support-stable}
Soit $\sigma=(\cP, Z)$ une précondition de stabilité sur la catégorie triangulée $\cD$. Soit $\Lambda$ un groupe abélien libre de type fini par lequel $Z$ se factorise, et soit~$Q$ une forme quadratique sur $\Lambda_\R$ définie négative sur le noyau de $Z$. Alors $\sigma$ satisfait la condition de support par rapport à $Q$ si et seulement si pour tout objet $\sigma$-stable de~$\cD$, on a $Q(v(A))\leq 0$.
\end{prop}

\begin{proof}
La nécessité de la condition est évidente. Montrons qu'elle est suffisante. 

Notons $p$ la projection orthogonale $\Lambda_\R\mapsto \Ker Z$, et écrivons, pour tout $v\in\Lambda$,
$$Q(v)=R(Z(v))-\Vert p(v)\Vert ^2,$$
où $R$ est une forme quadratique sur $\R^2\simeq\C$, et $\Vert .\Vert $ est la norme euclidienne induite par~$-Q$ sur~$\Ker Z$.

Soit $A$ un objet $\sigma$-semistable de $\cD$, et soient $A_1, \ldots, A_n$ ses facteurs de Jordan-Hölder. Par hypothèse, les $Q(v(A_i))$ sont tous positifs, donc les $R(Z(A_i))$ aussi.

On a 
$$R(Z(A))=\Big(\sum_i \sqrt{R(Z(A_i))}\Big)^2$$
car les $Z(A_i)$ sont tous sur la même demi-droite passant par l'origine, et ont pour somme~$Z(A)$. De plus, les $A_i$ étant $\sigma$-stables, ce qui précède montre
$$\sum_i \sqrt{R(Z(A_i))}\geq \sum_i \Vert p(v(A_i))\Vert \geq \Vert p(v(A))\Vert $$
par inégalité triangulaire, ce qui prouve bien $Q(v(A))\geq 0.$
\end{proof}

\begin{prop}\label{proposition:support-discret}
Soit $\cD$ une catégorie triangulée, et soit $\sigma=(\cP, Z)$ une précondition de stabilité sur $\cD$ satisfaisant la condition de support. Alors l'image par $Z$ de l'ensemble des objets $\sigma$-semistables de $\cD$ est un sous-ensemble discret de $\C$.
\end{prop}

\begin{proof}
Si l'ensemble en question n'est pas discret, la propriété de support permet de trouver un ensemble infini de $v\in\Lambda$ dans un compact de $\Lambda_\R$, ce qui est bien sûr impossible.
\end{proof}

\bigskip

Voici une première conséquence de la condition de support. Rappelons qu'une catégorie abélienne $\cA$ est de longueur finie si toute suite de monomorphismes dans $\cA$ $A_1\leftarrow A_2\leftarrow\cdots\leftarrow A_i\leftarrow\cdots$ est stationnaire, ainsi que toute suite d'épimorphismes $A_1\ra A_2\ra \cdots \ra A_i\ra\cdots$.

\begin{prop}
Soit $\sigma=(\cP, Z)$ une condition de stabilité sur une catégorie triangulée $\cD$. Pour tout $\phi\in\R$, la catégorie abélienne $\cP(\phi)$ est de longueur finie. 
\end{prop}

\begin{proof}
Supposons par exemple donnée une suite infinie d'inclusions strictes 
$$\cdots\subset A_i\subset\cdots\subset A_1$$
dans laquelle les $A_i$ sont tous semistables de phase $\phi$. Soit $v_i=v(A_i)\in\Lambda$. Alors on peut écrire $v_i=r_ie^{i\pi\phi}$ avec $r_i>0$. La suite exacte
$$0\ra A_{i+1}\ra A_i\ra A_i/A_{i+1}\ra 0$$
dans la catégorie abélienne $\cP(\phi)$ nous permet d'écrire 
$$r_ie^{i\pi\phi}=r_{i+1}e^{i\pi\phi}+Z(A_i/A_{i+1}),$$
avec $Z( A_i/A_{i+1})\in \R^*_+e^{i\pi\phi}$, ce qui montre que la suite $(r_i)_{i\geq 1}$ est strictement décroissante, et que par conséquent la suite des $Z(A_i/A_{i+1})$ converge vers $0$. D'après la proposition~\ref{proposition:support-discret}, la suite des $Z(A_i/A_{i+1})$ est stationnaire à $0$, contradiction.
\end{proof}

La proposition précédente garantit l'existence de filtrations de Jordan-Hölder dans $\cP(\phi)$ : tout objet semistable de phase $\phi$ a une filtration par des sous-objets semistables de phase $\phi$ dont les quotients successifs sont stables.

Introduisons une dernière condition de finitude parfois utile.

\begin{defi}
Soit $\sigma$ une condition de stabilité sur une catégorie triangulée $\cD$. On dit que $\sigma$ est rationnelle si la charge centrale de $\sigma$ est à valeurs dans $\Q[i]$.
\end{defi}

\section{L'espace des conditions de stabilité}

Dans cette section, nous fixons une catégorie triangulée $\cD$ et un morphisme de groupes abéliens $v : K_0(\cD)\ra\Lambda$, où $\Lambda$ est un groupe abélien libre de type fini.

\subsection{Préliminaires topologiques}

Soit $\mathrm{Dec}(\cD)$ l'ensemble des découpages de $\cD$. Si $\cP$ est un élément de $\mathrm{Dec}(\cD)$, et si $A$ est un objet de $\cD$, on note $\phi_\cP^-(A)$ et $\phi^+_\cP(A)$ la plus petite et la plus grande phase respectivement qui apparaissent dans la filtration de Harder-Narasimhan de $A$. On munit l'ensemble $\mathrm{Dec}(\cD)$ d'une distance $d$ --- à valeurs dans $[0, +\infty]$ --- par la formule
$$d(\cP, \cQ)=\sup_{A\in \mathrm{Ob}(\cD)}\{|\phi^+_\cP(A)-\phi^+_\cQ(A)|, |\phi^-_\cP(A)-\phi^-_\cQ(A)|\}.$$
C'est bien légitime car $d(\cP, \cQ)=0$ implique que pour tout réel $\phi$, et tout $A\in \mathrm{Ob}(\cP(\phi))$, on a $\phi^+_\cQ(A)=\phi^-_\cQ(A)=\phi$, soit $A\in \mathrm{Ob}(\cQ(\phi))$, et $\cP=\cQ$. On peut réécrire la distance~$d$ comme suit, voir \cite[Lemma 6.1]{Bridgeland07}.

\begin{lemm}\label{lemme: distance-alt}
Soient $\cP$ et $\cQ$ deux découpages de $\cD$. On a 
$$d(\cP, \cQ)=\inf\{\varepsilon\geq 0, \forall\phi\in\R, \mathrm{Ob}(\cQ(\phi))\subset\mathrm{Ob}(\cP([\phi-\varepsilon, \phi+\varepsilon]))\}.$$
\end{lemm}

Soit maintenant $\Stab(\cD)$ l'ensemble des conditions de stabilité sur $\cD$ qui se factorisent par $v$. On dispose de deux projections naturelles 
$$p : \Stab(\cD)\ra \mathrm{Dec}(\cD), (\cP, Z)\mapsto \cP$$
et 
$$q : \Stab(\cD)\ra\Hom(\Lambda, \C), (\cP, Z)\mapsto Z.$$
Dans tout ce qui suit, nous munissons $\Stab(\cD)$ de la topologie la plus fine qui rende les applications $p$ et $q$ continues. 

Les deux projections font de $\Stab(\cD)$ un sous-espace de $\mathrm{Dec}(\cD)\times \Hom(\Lambda, \C)$, qui est muni de la topologie produit. En particulier, $\Stab(\cD)$ est un espace topologique séparé.

Le résultat suivant est élémentaire. 

\begin{prop}
Les actions naturelles des groupes $\widetilde{GL}_2^+(\R)$ et $\mathrm{Aut}_\Lambda(\cD)$ sur $\Stab(\cD)$ sont continues. 
\end{prop}

\begin{rema}
On peut munir $\Stab(\cD)$ d'une métrique naturelle pour laquelle $\mathrm{Aut}_\Lambda(\cD)$ agit par isométries.
\end{rema}

\subsection{Le théorème de déformation de Bridgeland}

L'énoncé ci-dessous est le théorème principal de \cite{Bridgeland07}, tel que raffiné dans \cite[Appendix A]{BayerMacriStellari16}. Nous suivons essentiellement Bayer \cite{Bayer16}.

\begin{theo}\label{theoreme:def-Bridgeland}
Soit $Q$ une forme quadratique sur $\Lambda_\R$, et soit $\sigma$ une condition de stabilité sur $\cD$ qui satisfait la condition de support par rapport à $Q$. Soit $Z$ la charge centrale associée à $\sigma$, soit $\widetilde U$ l'ouvert de $\Hom(\Lambda, \C)$ constitué des $Z ' : \Lambda\ra\C$ tels que $Q$ est définie négative sur le noyau de $Z'$, et soit $U$ la composante connexe de $\widetilde U$ contenant~$Z$.

Soit enfin $V$ la composante connexe de $q^{-1}(U)$ contenant $\sigma$. Alors :
\begin{enumerate}
\item La restriction à $V$ de la projection naturelle $q : \Stab(\cD)\ra\Hom(\Lambda, \C), (\cP, Z)\mapsto Z$ est un revêtement de $U$;
\item Toute condition de stabilité dans $V$ vérifie la condition de support par rapport à~$Q$.
\end{enumerate}
\end{theo}

En particulier, l'espace topologique $\Stab(\cD)$ est muni d'une structure naturelle de variété complexe, qui rend $q$ localement biholomorphe.

\bigskip

Donnons les grandes lignes de la démonstration. Nous supposerons dans la suite --- c'est le cas non dégénéré --- que $Q$ est de signature $(2, \mathrm{rg}\Lambda-2)$. En particulier, si $Z'$ est un élément de $U$, alors $Z_\R : \Lambda_\R\ra\C$ est surjective.

On commence par un énoncé d'injectivité locale. 

\begin{lemm}\label{lemme:inj-loc}
Soient $\cP$ et $\cQ$ deux découpages de $\cD$ tels que $d(\cP, \cQ)<1$. Si $(\cP, Z)$ et $(\cQ, Z)$ sont deux conditions de stabilité sur $\cD$, alors $\cP=\cQ$. En particulier, la projection $q : \Stab(\cD)\ra\Hom(\Lambda, \C)$ est localement injective.
\end{lemm}

\begin{proof}
Soit~$\phi\in\R$, et supposons par exemple que $\cP(\phi)$ est non nulle. Soit~$A$ un objet non nul de $\cP(\phi)$. On veut montrer que $A$ est un objet de $\cQ(\phi)$, raisonnons par l'absurde et supposons que ce n'est pas le cas.

Comme $d(\cP, \cQ)<1$, on a $\phi-1<\phi^-_\cQ(A)\leq \phi^+_\cQ(A)<\phi+1.$ Si $\phi^-_\cQ(A)> \phi$, alors $Z(A)\in \R^*_+e^{i\pi\phi}$ est somme de réels de la forme $r_i e^{i\pi\phi_i}$ avec $r_i>0, \phi<\phi_i<\phi+1$, contradiction. On a donc, par un argument symétrique, 
$$\phi-1<\phi^-_\cQ(A)\leq\phi\leq \phi^+_\cQ(A)<\phi+1.$$
Supposons par exemple $\phi^+_\cQ(A)>\phi$. On peut trouver un triangle distingué
$$A_1\ra A\ra A_2\ra A_1[1]$$
dans lequel $A_1$ est un objet non nul de $\cQ(\mathopen]\phi, \phi+1\mathclose[)$ et $A_2$ est un objet non nul de $\cQ(\mathopen]\phi-1, \phi]).$ Comme plus haut, on montre que $\psi :=\phi^+_\cP(A_1)>\phi$. Il existe donc un objet non nul $B$ de $\cP(\psi)$, et un morphisme non nul $B\ra A_1$.

Comme $A$ est dans $\cP(\phi)$, le morphisme composé $B\ra A_1\ra A$ est nul, donc $B\ra A_1$ se factorise par un morphisme $B\ra A_2[-1]$, contradiction car $A_2[-1]$ est un objet de $\cQ(\leq \phi-1)$, donc de $\cP(\leq\phi)$.
\end{proof}

On va maintenant construire des sections locales de $q$. Le lemme suivant est clair.

\begin{lemm}\label{lemme:normalisation}
Il existe un élément $g$ de $GL_2^+(\R)$ tel que pour tout $v\in\Lambda$, on ait 
$$Q(v)=|g(Z(v))|^2+Q(p(v)),$$
où $p$ est la projection orthogonale de $\Lambda_\R$ sur $\Ker Z$.
\end{lemm}

%\begin{lemm}
%Soit $(\cP_t, Z_t)$ une famille à un paramètre de conditions de stabilité, et soit $A$ un objet de $\cD$. Supposons que $A$ est $\sigma_1$-semistable, mais pas $\sigma_0$-semistable. Alors il existe 
%\end{lemm}

%Commençons par une réduction. 
%
%\begin{lemm}\label{lemme:normalisation}
%Pour démontrer le théorème, on peut supposer qu'il existe une norme euclidienne $\Vert .\Vert $ sur $\Ker Z$ telle que pour tout $v\in\Lambda$, on ait 
%$$Q(v)=|Z(v)|^2-\Vert p(v)\Vert ^2,$$
%où $p$ est la projection orthogonale de $\Lambda_\R$ sur $\Ker Z$.
%\end{lemm}
%
%\begin{proof}
%On fait agir le groupe $\widetilde{GL}_2^+(\R)$ sur $Stab(\cD)$. Il suffit donc de montrer qu'il existe une norme euclidienne $\Vert .\Vert $ sur $\Ker Z$ et un élément $g$ de $GL_2^+(\R)$ tel que 
%$$Q(v)=|g(Z(v))|^2-\Vert p(v)\Vert ^2.$$
%Choisir $g$ de telle sorte que $((\Ker Z)^\perp, Q)\ra (\C, |.|), v\mapsto g(Z(v))$ est une isométrie prouve le résultat.
%\end{proof}

%\begin{rema}
%La condition du lemme précédent est préservée par action de la préimage du groupe des rotations du plan dans $\widetilde{GL}_2^+(\R)$.
%\end{rema}

Avec les notations du lemme précédent, supposons que l'on puisse prendre $g=1$. 

\begin{lemm}\label{lemme:desc-W}
Soit $W$ l'ensemble des éléments $Z'$ de $U$ tels que $Z_{|\Ker Z^\perp}=Z'_{|\Ker Z^\perp}$. Alors $Z'$ est dans $W$ si et seulement si $Z'=Z+u\circ p$ , où $u$ est un morphisme de $\Ker Z$ dans $\C$ tel que $\Vert u\Vert <1$.
\end{lemm}

Ici, $\Vert .\Vert $ est la norme d'opérateur, $\Ker Z$ étant muni de la norme $-Q$ et $\C$ de la valeur absolue usuelle.

\begin{proof}
Pour vérifier cette assertion, si $\Vert u\Vert <1$ et si $v\in\Lambda\setminus\{0\}$ est tel que $Z(v)+u\circ p(v)=0$, on a $Z(v)=-u(p(v))$, donc 
$$Q(v)=|Z(v)|^2+Q(p(v))=|u(p(v))|^2+Q(p(v))<0.$$
Réciproquement, si $Z'$ est dans $W$, $Z'-Z$ s'annule sur $\Ker Z^\perp$, donc s'écrit $u\circ p$ comme plus haut. Si $\Vert u\Vert \geq 1,$ on peut trouver $\alpha\in\Ker Z_\R$ tel que $Q(\alpha)=-1$ et $|u(\alpha)|\geq 1.$ On a supposé que $Z_\R$ est une surjection. Soit donc $\beta\in \Ker Z^\perp$ tel que $Z(\beta)=-u(\alpha)$. Alors 
$$Z'(\alpha+\beta)=0$$ 
et 
$$Q(\alpha+\beta)=-1+|Z(\beta)|^2=|u(\alpha)|^2-1\geq 0,$$
contradiction.
\end{proof}

Soit $W_\R$ l'ensemble des éléments $Z+u\circ p$ de $W$ tels que $u\in\Hom(\Ker Z, \R)$. Le résultat suivant contient le calcul clé.

\begin{lemm}\label{lemme:calcul-cle}
On garde les notations précédentes.
\begin{enumerate}
\item On peut trouver un voisinage $V$ de $\sigma$ dans $\Stab(\cD)$ tel que la restriction de $q$ à $V\cap q^{-1}(W_\R)$ est un homéomorphisme sur $W_\R$.
\item Toute condition de stabilité dans $V\cap q^{-1}(W_\R)$ vérifie la condition de support par rapport à $Q$.
\end{enumerate}
\end{lemm}

\begin{proof}
Écrivons $\sigma=(\cP, Z)$, et soit $\cA$ le c\oe{}ur de $\cP$. Si $u$ est un élément de~$W_\R$, soit $\sigma_u=(\cA, Z+u\circ p)$. On vérifie immédiatement que $Z+u\circ p$ est une fonction de stabilité. On vérifie ensuite que $\sigma_u$ est une précondition de stabilité, i.e. que les filtrations de Harder-Narasimhan associées à $\sigma_u$ existent. Nous renvoyons à \cite[4.6 à 4.10]{Bayer16} pour une preuve détaillée. Le point-clé, qui suit de la condition de support, est de montrer que pour tout objet $A$ de $\cA$, et pour tout $C\in\R$, il n'existe qu'un nombre fini d'éléments de $\Lambda$ de la forme $v(B)$, où $B$ est un sous-objet de $A$ et $\Re(Z+u\circ p)(B)<C.$

\bigskip

Si $t\in[0,1]$, notons $Z_t=Z+t(u\circ p)$, et $\sigma_t=(\cA, Z_t)$, de sorte que $\sigma_0=\sigma$ et $\sigma_1=\sigma_u$. On note $\phi_t$ la phase correspondant à $\sigma_t$. Montrons maintenant que $\sigma_1=\sigma_u$ vérifie la condition de support par rapport à $Q$. On raisonne par l'absurde et l'on suppose qu'il existe un objet $A$ de $\cA$ qui est $\sigma_1$-semistable mais vérifie $Q(v(A))<0$. Alors $A$ n'est pas ($\sigma=\sigma_0$)-semistable. 

On va appliquer un premier exemple de wall-crossing. Les estimées mentionnées au premier paragraphe de la preuve permettent de montrer qu'il n'existe qu'un nombre fini d'éléments de $\Lambda$ de la forme $v(B)$, où $B$ est un sous-objet strict non-nul de $A$ tel qu'il existe $t\in[0,1]$ avec $\phi_t(B)\geq\phi_t(A).$ Pour chacun de ces $v(B)$, comme $\phi_0(B)\leq\phi_0(A)$, il existe par continuité un~$t'\in[0,1]$ maximal tel que $\phi_{t'}(B)=\phi_{t'}(A)$. Soit $t_0$ le plus grand de ces $t'$. Alors par construction, $A$ est $\sigma_{t_0}$-semistable, mais n'est pas stable, admettant un sous-objet strict~$B$ tel que $\phi_{t_0}(B)=\phi_{t_0}(A)$. 

L'argument de finitude des $v(B)$ montre en particulier que $A$ admet une filtration de Jordan-Hölder pour la condition de stabilité $\sigma_{t_0}$. Le calcul de la preuve de la proposition~\ref{proposition:support-stable}, utilisant cette filtration et la normalisation du lemme \ref{lemme:normalisation}, nous donne un sous-quotient strict $A_1=B_1/C_1$ de $A=A_0$ tel que 
$$\phi_{t_0}(A_1)=\phi_{t_0}(B_1)=\phi_{t_0}(C_1)=\phi_{t_0}(A)$$
et $Q(v(A_1))<0$. 

Par récurrence, on trouve des suites
$$0=C_0\subset C_1\subset \cdots\,\,\mathrm{et}\,\, A=B_0\supset B_1\supset\cdots,$$
et une suite décroissante $t_i$ d'éléments de $[0,1]$ tels que $A_i:=B_i/C_i$ est un sous-quotient strict, $\sigma_{t_i}$-semistable, de $A_{i-1}$, $Q(v(A_i))<0$ et 
$$\phi_{t_{i}}(A_{i+1})=\phi_{t_{i}}(B_{i+1})=\phi_{t_i}(C_{i+1})=\phi_{t_i}(A_i)$$
et $Q(v(A_{i+1}))<0.$

%une suite $(A_i)_{i\geq 0}$ d'objets de $\cA$, tels que $A_i$ est un sous-quotient strict de $A_j$ pour $i>j$, et une suite décroissante $t_i$ d'éléments de $[0,1]$ tels que $A_i$ est $\sigma_{t}$-semistable pour $t\geq t_i$, $\phi_{t_i}(A_i)=\phi_{t_i}(A_{i+1})$ et $Q(v(A_i))<0.$

On a $\phi_{t_0}(B_1)=\phi_{t_0}(A)$ et, par $\sigma_1$-semistabilité de $A$, $\phi_{1}(B_1)\leq\phi_{1}(A)$. La linéarité des fonctions $Z_t$ montre alors l'inégalité
$$\phi_{t}(B_1)\geq\phi_{t}(A)$$
pour tout $t\leq t_0$. Une récurrence immédiate donne
$$\forall i\geq 0, \phi_{t_i}(B_i)\geq \phi_{t_i}(A).$$
D'après ce qui précède, les $v(B_i)$ sont en nombre fini, ainsi donc que les $Z(B_i)$. Cependant, si $B_{i+1}$ est un sous-objet strict de $B_i$, alors \[Z(B_i)-Z(B_{i+1})=Z(B_i/B_{i+1})\in\mathbb H\cup \R^*_-.\] Cela implique que la suite $(B_i)$ est stationnaire. De même, la suite $(C_i)$ est stationnaire, ainsi que la suite $(A_i)$, contradiction : $\sigma_1$ vérifie la condition de support par rapport à~$Q$.

\bigskip

Ce qui précède nous fournit une fonction 
$$W_\R\ra \Stab(\cD), u\mapsto (\cA, Z+u\circ p).$$
Montrons-en la continuité. Soit en effet $\phi$ un réel, et soit $A$ un élément $\sigma_u$-semistable de~$\cA$, de phase $\phi$. Grâce au lemme \ref{lemme: distance-alt}, et à décalage près, il s'agit de montrer que $\phi_0^+(A)$ et $\phi_0^-(A)$ sont proches de $\phi$ si $u$ est proche de $0$ dans $W_\R$. 

Soit $A_+$ le sous-objet semistable maximal de $A$ tel que $\phi_0^+(A)=\phi_0(A_+)$. Comme $A$ est $\sigma_1$-semistable, on a $\phi_1(A_+)\leq \phi$. Si $u$ est suffisamment proche de $0$, la continuité de la phase en fonction de $u$ nous garantit que l'on a $\phi_0^+(A)=\phi_0(A_+)\leq \phi_u(A_+)+\eta\leq\phi+\eta$ pour un $\eta>0$ petit, ne dépendant que de $u$. De même, on a dans ce cas $\phi_0^-(A)\geq \phi-\eta.$ On a finalement, pour $\eta$ tendant vers $0$ avec $u$, 
$$\phi-\eta\leq\phi_0^-(A)\leq\phi_0^+(A)\leq \phi+\eta,$$
ce qui conclut.

Les énoncés que l'on vient de montrer, avec l'injectivité locale, prouvent le résultat.
\end{proof}

Voici une version faible du théorème, qui contient essentiellement tous les arguments nécessaires à la preuve complète.

\begin{lemm}\label{lemme:homeo-local}
L'application $q$ est un homéomorphisme local. 
\end{lemm}

\begin{proof}
Il suffit de montrer le résultat au voisinage de $\sigma$, et l'on peut supposer quitte à faire agir $GL_2^+(\R)$ que l'on peut prendre $g=1$ dans le lemme \ref{lemme:normalisation}.

Pour montrer le résultat, il faut montrer, vue l'injectivité locale, que si $Z'\in\Hom(\Lambda, \C)$ est suffisamment proche de $Z$, alors $Z'$ est la charge centrale d'une condition de stabilité $\tau$ proche de $\sigma$.

Comme $Z'$ est proche de $Z$, on peut choisir $g\in GL_2^+(\R)$, proche de $1$, tel que $g\circ Z'$ et~$Z$ coïncident sur $(\Ker Z)^\perp$. On peut remplacer $Z'$ par $g\circ Z'$ et supposer que $Z'$ s'écrit
$$Z'=Z+u\circ p + i v\circ p,$$
où $p$ est la projection orthogonale de $\Lambda_\R$ sur $\Ker Z$ et $u, v\in \Hom(\Ker Z, \R)$ sont proches de $0$. Notons $Z_1=Z+u\circ p$. Alors $Z_1$ est proche de $Z$ et le lemme ci-dessus fournit une condition de stabilité $\sigma_1$ de charge centrale~$Z_1$, proche de~$\sigma$. On peut remplacer $Z$ par $Z_1$, $\sigma$ par $\sigma_1$, et supposer donc 
$$Z'=Z+ i\, v\circ p',$$
où $p'$ est proche de $p$. Quitte à encore remplacer $Z'$ par $g\circ Z'$ avec $g$ proche de $1$, on peut supposer $p'=p$.

On a
$$-iZ'=-iZ+v\circ p.$$
Notant $T$ un relèvement de la multiplication par $i$ dans $\widetilde{GL}_2^+(\R)$, le lemme précédent fournit une condition de stabilité $-T\tau$, proche de $-T\sigma$, de charche centrale $-iZ'$. Alors $\tau$ est bien une condition de stabilité proche de $\sigma$, de charge centrale $Z'$.
\end{proof}

Soit $V$ la composante connexe de $p^{-1}(U)$ contenant $Z$. On peut adapter l'argument précédent pour montrer la surjectivité de $q_{|V} : V\ra U$. Mieux, on peut montrer de même la surjectivité locale de $q_{|V}$, ce qui implique par \cite[Proposition 5.6]{Huybrechts12} que $q_{|V}$ est bien un revêtement de $U$. 

Enfin, les arguments ci-dessus, joints à la seconde propriété du lemme \ref{lemme:calcul-cle} et à la remarque \ref{remarque:support-action}, montrent que les conditions de stabilité dans $V$ vérifient la condition de support par rapport à $Q$.

La remarque \ref{remarque:uniformite}, jointe au second énoncé du théorème, a la conséquence importante suivante qui traduit l'uniformité de la condition de support.

\begin{coro}\label{corollaire:uniformite-support}
Soit $\Vert .\Vert $ une norme sur $\Lambda$, et soit $\sigma$ une condition de stabilité sur~$\cD$. Alors il existe un voisinage ouvert $U$ de $\sigma$ dans $\Stab(\cD)$, et une constante $C>0$, tels que pour tout $\tau=(\cQ, Z')\in U$, et pour tout objet $A$ de $\cD$ qui est $\tau$-semistable, on ait
$$|Z(A)|\geq C\Vert v(A)\Vert .$$
\end{coro}

\subsection{Murs et chambres}

On décrit maintenant la variation des ensembles d'objets (semi)stables dans $\cD$ quand la condition de stabilité varie continûment. Dans ce qui suit, nous noterons $\cA_\sigma$ la catégorie abélienne c\oe{}ur d'une condition de stabilité $\sigma$.

Voici d'abord un lemme catégorique.

\begin{lemm}
Soit $\varepsilon>0$. Il existe un recouvrement de $\Stab(\cD)$ par des ouverts $U$ tels que si $\sigma=(\cP, Z)$ est dans $U$, et si $A$ est un objet de $\cP(\mathopen]\varepsilon,1-\varepsilon\mathclose[)$, alors pour tout $\tau\in U$, $A$ est un objet de $\cA_\tau$. 

Pour un tel ouvert $U$, et $\sigma=(\cP, Z)\in U$, si $f : B\ra C$ est un monomorphisme entre objets de $\cP(\mathopen]\varepsilon,1-\varepsilon\mathclose[)$, vus comme objets de $\cA_\sigma$, dont le conoyau est dans $\cP(\mathopen]\varepsilon,1-\varepsilon\mathclose[)$, alors, pour tout $\tau\in U$, $f$ est un monomorphisme.
\end{lemm}

\begin{proof}
La définition de la topologie de $\Stab(\cD)$ montre que les fonctions 
$$\sigma\mapsto \phi^-_\sigma(A), \phi^+_\sigma(A)$$
sont toutes les deux continues, ce qui montre le premier point.

Pour prouver le second point, remarquons que $f$ est un monomorphisme dans $\cA_{\sigma}$ si et seulement si le cône de $f$ est dans $\cD^{\geq 0}=\cP(\leq 1)$ --- on utilise ici la t-structure de la proposition \ref{proposition:t-structure} associée à $\cA_\sigma$. Dans ce cas, le cône de $f$ dans $\cD$ est isomorphe au conoyau $C$ de $f$ dans $\cA_\sigma$, ce qui permet de conclure.
\end{proof}

\begin{prop}\label{proposition:topologie-stable}
Soit $E$ un objet de $\cD$. L'ensemble des $\sigma\in \Stab(\cD)$ tel que $E$ est $\sigma$-semistable (resp. stable) est fermé (resp. ouvert) dans $\Stab(\cD)$.
\end{prop}

\begin{proof}
Les fonctions $\sigma\mapsto \phi^-_\sigma(A), \phi^+_\sigma(A)$ sont toutes les deux continues. Par ailleurs, $A$ est $\sigma$-semistable si et seulement si $\phi^-_\sigma(A)=\phi^+_\sigma(A)$. Cela montre que l'ensemble des $\sigma$ tel que $A$ est $\sigma$-semistable est fermé dans $\Stab(\cD)$. 

Soit $\sigma=(\cP, Z)$ une condition de stabilité sur~$\cD$, et soit $A$ un objet $\sigma$-stable de $\cD$.
Supposons pour fixer les idées --- quitte à faire agir $\widetilde{GL}_2^+(\R)$ --- que la phase~$\phi_\sigma(A)$ est~$1/2$. Soit $v=v(A)\in \Lambda$. Soit $\cA$ le c\oe{}ur de $\cP$, et soit $U$ un voisinage ouvert de $\sigma$ dans $\Stab(\cD)$ tel que pour tout $\tau=(\cQ, Z')\in U$, $A$ est un objet de $\cQ(\mathopen]1/3, 2/3\mathclose[)$. 

Soit $W$ l'ensemble des éléments $w$ de $\Lambda$ tels que
$$\exists (\cQ, Z_\tau)\in U, \quad 0\leq \mathrm{Im} Z_\tau(w)\leq \mathrm{Im} Z_\tau(v)$$
et 
$$w\in \R^*_+e^{i\pi\phi},$$ 
avec $\phi\in \mathopen]1/3, 2/3\mathclose[.$
Quitte à retrécir $U$, on peut supposer que l'ensemble des $|Z(w)|,$ \mbox{$w\in W$}, est borné. Soit $W_d$ l'ensemble des $w\in W$ tels qu'il existe $\tau=(\cQ, Z_\tau)\in U,$ et un monomorphisme $B\ra A$ d'objets de $\cQ(\mathopen]1/3, 2/3\mathclose[)$, vus comme objets de $\cA_\tau$, de conoyau dans $\cQ(\mathopen]1/3, 2/3\mathclose[)$, avec $B$ $\tau$-semistable et $v(B)=w$. Le corollaire \ref{corollaire:uniformite-support} garantit que $W_d$ est un ensemble fini, quitte à encore retrécir $U$.

Soit $\tau=(\cQ, Z_\tau)\in U$, et supposons que $A$ n'est pas $\tau$-stable. Alors on peut trouver \mbox{$w\in W_d$,} avec $\mathrm{Im}(Z_\tau(w)/ Z_\tau(v))\geq 0$, un monomorphisme $B\ra A$ d'objets de $\cQ(\mathopen]1/3, 2/3\mathclose[)$, vus comme objets de $\cA_\tau$, de conoyau dans $\cQ(\mathopen]1/3, 2/3\mathclose[)$, avec $B$ \mbox{$\tau$-semistable} et \mbox{$v(B)=w$} --- prendre $B$ maximal semistable de phase $\phi^+_\tau(A)$.

Soit maintenant $\tau'=(\cQ', Z_{\tau'})\in U$ tel que $\mathrm{Im}(Z_{\tau'}(w)/Z_{\tau'}(v))\geq 0$ --- c'est une condition fermée dans $U$. Le lemme précédent montre que pour tout tel $\tau'$, le morphisme $B\ra A$ est un monomorphisme dans $\cA_{\tau'}$. On a par ailleurs $\phi_{\tau'}(B)\geq \phi_{\tau'}(A)$, donc $A$ n'est pas $\tau$-stable.

Finalement, pour chaque $w\in W_d$, on a construit un fermé $F_w$ de $U$ (vide ou défini par l'équation $\mathrm{Im}(Z_\tau(w)/Z_\tau(v))\geq 0$) de telle sorte que l'ensemble des $\tau\in U$ tels que $A$ n'est pas $\tau$-stable est la réunion des $F_d$. Cela montre bien que l'ensemble des $\tau\in U$ tels que $A$ est $\tau$-stable est ouvert.
\end{proof}

La proposition précédente donne en fait une description précise du lieu où un objet $A$ peut être déstabilisé. Reprenant l'exact argument de la preuve, nous obtenons le résultat suivant, dû essentiellement à Bridgeland \cite{Bridgeland08} et Toda \cite{Toda08}. Nous reprenons essentiellement les énoncés de \cite[Proposition 3.3]{BayerMacri11} et \cite[Proposition 2.3]{BayerMacri14proj}.

\begin{prop}\label{proposition:murs-chambres}
Soit $v$ un élément de $\Lambda$. Il existe un ensemble $\cW$ de sous-variétés à bord de codimension $1$ dans $\Stab(\cD)$, appelées \emph{murs}, tels que, appelant \emph{chambres} les composantes connexes du complémentaire de l'ensemble des $W\in\cW$, on ait les propriétés suivantes : 
\begin{enumerate}
\item $\cW$ est une famille localement finie : tout compact de $\Stab(\cD)$ n'intersecte qu'un nombre fini de murs;
\item Soit $W\in\cW$. On peut trouver $w\in\Lambda$ non proportionnel à $v$ tel que $W$ est l'adhérence d'un ouvert de $\{\sigma=(\cP, Z)\in \Stab(\cD), Z(w)/ Z(v)\in\R^*\}$. De plus, pour tout $\sigma=(\cP, Z)\in W$, on peut trouver $\phi\in\R$, et un monomorphisme $B\ra A$ dans $\cP(\phi)$, tels que $v(A)=v$ et $v(B)=w$;
\item Si $C$ est une chambre, $A$ un objet de $\cD$, et $\sigma, \tau\in C$ (resp. $\overline C$), alors $A$ est $\sigma$-stable (resp. semistable) si et seulement si $A$ est $\tau$-stable (resp. semistable);
\item Si $W$ est un mur, il existe un objet $A$ de $\cD$ qui est semistable sur $W$, et instable dans une des chambres adjacentes\footnote{i.e., dont l'adhérence intersecte $W$.} à $W$. 
\end{enumerate}
De plus, si $(W_i)_{i\in I}$ est un ensemble fini de murs, la famille des $W\cap \bigcap_{i\in I}W_i, W\in \cW$, définit une famille de murs et de chambres dans $\bigcap_{i\in I}W_i$ satisfaisant les propriétés précédentes.
\end{prop}

\begin{rema}
L'action de $\widetilde{GL}_2^+(\R)$ laisse invariants les murs --- et les chambres --- décrits ci-dessus. En effet, elle préserve l'ensemble des catégories $\cP(\phi), \phi\in\R$, donc les ensembles d'objets stables et semistables.
\end{rema}

\bigskip

Dans la situation précédente, une \emph{strate} est une composante connexe de l'intersection d'un ensemble fini de murs. Voici une remarque utile, dont la démonstration est évidente.

\begin{prop}\label{proposition:rationnel-dense}
Pour tout $v\in\Lambda$, l'ensemble des conditions de stabilité rationnelles est dense dans chaque strate de la décomposition en murs et chambres de $\Stab(\cD)$ associée à $v$.
\end{prop}

%\begin{prop}
%Soit $v$ un élément de $\Lambda$. Si $\sigma=(\cP, Z)$ est une condition de stabilité sur $\cD$, notons $M_\sigma(v)$ l'ensemble des éléments $w$ de $\Lambda$ tel qu'il existe un objet $\sigma$-semistable $A$ de $\cD$, de phase $\phi$, et un sous-objet $B$ de $A$ dans $\cP(\phi)$, tels que $v(A)=v$ et $v(B)=w$.
%
%Soient $\sigma$ et $\tau$ deux conditions de stabilité sur $\cD$. Notons aussi $D_\sigma(v)$ l'ensemble des éléments $w$ de $\Lambda$ tel qu'il existe un objet $\sigma$-semistable $A$ de $\cD$, tel que si $B$ est l'objet 
%
%Alors pour tout $\sigma\in Stab(\cD)$ il existe un voisinage ouvert $U$ de $\sigma$ dans $Stab(\cD)$, tel que $\bigcup_{\tau\in U} M_\tau(v)$ est un sous-ensemble fini de $\Lambda$. De plus, on peut choisir $U$ et un sous-ensemble fini $D_U(v)$ de $\Lambda$ de telle sorte que si $A$ est un objet $\sigma$-semistable $A$ de $\cD$, de phase $\phi$, et si $A$ n'est pas $\tau$-semistable 
%\end{prop}
%
%\begin{proof}
%Facile
%\end{proof}
%

\section{Exemples}

Dans la suite de ces notes, nous allons appliquer la théorie générale décrite ci-dessus à des situations venant de la géométrie. 

Soit $X$ une variété complexe. On note $D^b(X)$ la catégorie dérivée bornée de la catégorie abélienne $\mathrm{Coh}(X)$ des faisceaux cohérents sur $X$.

Supposons $X$ projective et lisse. Si $A$ et $B$ sont deux objets de $D^b(X)$, alors pour tout entier $i$, l'espace $\Hom(A, B[i])$ est de dimension finie sur $\C$, et il est nul si $i$ est assez grand en valeur absolue. La caractéristique d'Euler
$$\chi(A, B) : = \sum_{i\in\Z} (-1)^i\dim \Hom(A, B[i])$$
est donc bien définie. Elle induit une forme bilinéaire sur le groupe de Grothendieck $K(X)=K_0(D^b(X)).$

Disons que $A$ et $A'$ sont \emph{numériquement équivalents}, et notons $A\sim A'$ si pour tout objet $B$ de $D^b(X)$, on a 
$$\chi(A, B)=\chi(A', B).$$
La relation $\sim$ est bien une relation d'équivalence : \emph{l'équivalence numérique}. Le théorème de dualité de Serre montre que la forme bilinéaire $\chi$ descend en une forme bilinéaire non dégénérée sur $K_{\mathrm{num}}(X):=K(X)/\sim$.

On sait --- c'est une conséquence du théorème de Hirzebruch-Riemann-Roch --- que le groupe $K_{\mathrm{num}}(X)$ est un groupe abélien libre de type fini. On le notera $\Lambda$. 

\begin{defi}
Une condition de stabilité numérique sur $X$ est une condition de stabilité sur $D^b(X)$ qui se factorise par la flèche naturelle $K_0(X)\ra K_{\mathrm{\mathrm{num}}}(X)$. On note $\Stab(X)$ l'espace des conditions de stabilité numériques sur $X$.
\end{defi}

Si $X$ est une courbe irréductible, la flèche 
$$K(X)\ra \Z^2, [A]\mapsto(-\deg A, \mathrm{rg}(A))$$
identifie $K_{\mathrm{num}}(X)$ à $\Z^2$.

Si $X$ est une surface irréductible, notons $\NS(X)$ le groupe des fibrés en droites sur $X$ modulo équivalence numérique. Alors $K_{\mathrm{num}}(X)\otimes\Q$ s'identifie à $\Q\oplus \NS(X)\oplus \Q$ via $E\mapsto \ch(E)$.

\bigskip

{\bf Dans la suite de ce texte, une condition de stabilité sera toujours une condition de stabilité numérique.}

\subsection{Courbes}

Le cas des courbes est parfaitement compris. Soit $X$ une courbe lisse, projective, irréductible de genre~$g$. On dispose sur la catégorie abélienne d'une fonction de stabilité 
$$Z : \mathrm{Coh}(X)\ra \C, A\mapsto -\deg A+i\,\mathrm{rg}(A).$$

L'existence des filtrations de Harder-Narasimhan pour $Z$ est classique. Par ailleurs, toute forme quadratique définie positive sur $\R^2=K_{\mathrm{num}}(X)_\R$ satisfait manifestement la condition de support pour $(\mathrm{Coh}(X), Z)$, puisque $Z$ est une injection. On a donc montré que $\sigma=(\mathrm{Coh}(X), Z)$ est une condition de stabilité sur $X$.

\bigskip

Soit $\Stab^0(X)$ la composante connexe de $\Stab(X)$ contenant la condition de stabilité~$\sigma$. Soit $U$ la composante connexe contenant $Z$ de l'ouvert de $\Hom(K(X), \C)$ constituée des applications $Z'$ telles que $Z'_\R : \R^2\ra \C$ est une injection. Alors $U$ est l'orbite de $Z$ sous l'action de $\mathrm{GL}_2^+(\R)$.

D'après le théorème \ref{theoreme:def-Bridgeland}, la préimage $V$ de $U$ dans $\Stab^0(X)$ est un revêtement de $U$. L'application $q : V\ra U$ est $\widetilde{GL}_2^+(\R)$-équivariante, et l'on vérifie immédiatement que l'action de $\widetilde{GL}_2^+(\R)$ sur $\Stab^0(X)$ est libre. $V$ est donc l'orbite de $\sigma$ sous $\widetilde{GL}_2^+(\R)$. En particulier, tout objet $\sigma$-semistable est $\sigma'$-semistable pour tout $\sigma'\in V$.

Supposons un moment que l'inclusion de $V$ dans $\Stab^0(X)$ soit stricte. Alors on peut trouver une condition de stabilité $\tau=(\cA, Z')$ sur $X$, appartenant au bord de $V$, dont la charge centrale $Z'$ n'est pas injective. Soit $Q$ une forme quadratique sur $\R^2=K(X)_\R$ qui satisfait la condition de support par rapport à $\sigma$. Alors $Q$ est strictement négative sur un cône ouvert de $\R^2$, et positive ou nulle sur tout $v\in\Z^2$ qui est la classe d'un objet $\sigma$-semistable. La fermeture de la semistabilité prouvée dans la proposition \ref{proposition:topologie-stable} et le paragraphe précédent montrent que tout objet $\sigma$-semistable est $\tau$-semistable.

Si le genre $g$ de $X$ est au moins $0$, i.e. si $X$ n'est pas isomorphe à $\mathbb P^1_\C$, on sait que tout couple $(-d, r)$ avec $r>0$ est de la forme $(-\deg E, \mathrm{rg} E)$, avec $E$ $\sigma$-semistable sur $X$. En particulier, tout cône ouvert de $\R^2$ rencontre l'ensemble des $v(E)$, où $E$ parcourt les faisceaux localement libres $\sigma$-semistables sur $X$. 

L'espace $\Stab^0(X)$ est donc un revêtement de $U\simeq GL_2^+(\R)$. L'application de revêtement correspondante 
$$\Stab^0(X)\ra GL_2^+(\R)$$
est $\widetilde{GL}_2^+(\R)$-équivariante, et l'on vérifie immédiatement que l'action de $\widetilde{GL}_2^+(\R)$ sur $\Stab^0(X)$ est libre. Il en résulte le résultat suivant.

\begin{prop}\label{proposition:stab-courbe}
Soit $X$ une courbe irréductible de genre positif ou nul, et soit $\Stab^0(X)$ la composante de $\Stab(X)$ contenant la condition de stabilité $\sigma$ définie ci-dessus. Alors $\Stab^0(X)$ est un espace principal homogène sous $\widetilde{\mathrm{GL}}_2^+(\R)$. Il est biholomorphe à $\C\times\mathbb H$.
\end{prop}

Il est en fait possible de prouver le résultat suivant, plus fort, \cite{Bridgeland07, Macri07}, qui montre en outre que $\Stab(X)$ est connexe.

\begin{theo}
Soit $X$ une courbe irréductible de genre positif ou nul. Alors $\Stab(X)$ est un espace principal homogène sous $\widetilde{\mathrm{GL}}_2^+(\R)$. On a de plus $\Stab(\mathbb P^1_\C)\simeq \C^2.$
\end{theo}

\subsection{Basculement}

Dans la suite de ce texte, nous allons concentrer notre attention sur la catégorie $D^b(X)$, où $X$ est de dimension $2$. Il est immédiat de remarquer que la stabilité via la pente ne peut pas donner de condition de stabilité de c\oe{}ur $\mathrm{Coh}(X)$ : les faisceaux supportés sur les points fermés seraient dans le noyau de la charge centrale. Plus généralement, il n'est pas très difficile de montrer que $\mathrm{Coh}(X)$ n'est le c\oe{}ur d'aucune condition de stabilité numérique sur $D^b(X)$.

Dans ce qui suit, nous montrons comment construire des catégories abéliennes intéressantes dans $D^b(X)$ par un procédé dit de \emph{basculement}. Cette construction est due à Happel-Reiten-Smal\o~\cite{HRS96}.

\begin{defi}
Soit $\cA$ une catégorie abélienne. Un couple de torsion sur $\cA$ est un couple de sous-catégories additives pleines $(\cF, \cT)$ de $\cA$ qui satisfait les deux propriétés suivantes :
\begin{enumerate}
\item Pour tous objets $F$ de $\cF$ et $T$ de $\cT$, on a $\Hom(T, F)=0$;
\item Pour tout objet $E$ de $\cA$, on peut trouver des objets $F$ de $\cF$ et $T$ de $\cT$, ainsi qu'une suite exacte
$$0\ra T\ra E\ra F\ra 0.$$
\end{enumerate}
\end{defi}

\begin{rema}\label{remarque:stabilite-quotient}
La propriété d'annulation garantit l'unicité de la suite exacte en 2. 

Par ailleurs, si $T$ est un objet de $\cT$ et $T\ra Q$ un quotient de $T$, la propriété 2 fournit une suite exacte 
$$0\ra T'\ra Q\ra F'\ra 0,$$
dans laquelle $T'$ et $F'$ sont des objets de $\cT$ et $\cF$ respectivement. La composition $T\ra Q\ra F'$ est un épimorphisme, nul par la propriété d'annulation, donc $F'$ est nul et $Q$ est un objet de $\cT$. De même, tout sous-objet d'un objet de $\cF$ est un objet de $\cF$.
\end{rema}

Il est agréable de se souvenir de la définition en pensant à l'exemple de $\mathrm{Coh}(X)$, où $X$ est une variété algébrique, $\cF$ est la sous-catégorie des faisceaux sans torsion, et $\cT$ la catégorie des faisceaux de torsion.

\bigskip

Ce qui suit n'est pas difficile, mais c'est un point clé.

\begin{prop}\label{proposition:tilt}
Soit $\cA$ une catégorie abélienne, c\oe{}ur d'une t-structure bornée sur une catégorie triangulée $\cD$, et soit $(\cF, \cT)$ un couple de torsion sur $\cA$. 

Soit $\cA^\sharp$ la sous-catégorie pleine de $\cD$ dont les objets sont les objets $E$ de $\cD$ tels que $H^{-1}(E)$ est un objet de $\cF$, $H^0(E)$ est un objet de $\cT$, et $H^i(E)=0$ si $i\neq -1, 0$, où les~$H^i$ sont calculés pour la t-structure de c\oe{}ur $\cA$.

Alors $\cA^\sharp$ est le c\oe{}ur d'une t-structure bornée sur $\cA$.
\end{prop}

\begin{defi}
La catégorie $\cA^\sharp$ est le basculement de $\cA$ par rapport au couple de torsion $(\cF, \cT)$.
\end{defi}

Voici une façon peut-être plus claire de comprendre ce qui précède : par construction, la catégorie $\cA^\sharp$ contient les catégories $\cT$ et $\cF[1]$, et tout objet de $\cA^\sharp$ est extension --- nécessairement unique --- d'un objet de $\cF[1]$ par un objet de $\cT$. Autrement dit, $(\cT, \cF[1])$ est un couple de torsion sur $\cA^\sharp$. Le basculement de $\cA^\sharp$ par rapport à $(\cT, \cF[1])$ est la catégorie abélienne $\cA[1]$.

Les suites exactes courtes dans $\cA^\sharp$ correspondent aux triangles distingués
$$A\ra B\ra C\ra A[1]$$
dans $\cD$ dont tous les termes sont des objets de $\cA^\sharp$. En particulier, si $A$ est un objet de~$\cA^\sharp$, on a une suite exacte dans $\cA^\sharp$
$$0\ra H^{-1}(A)[1]\ra A\ra H^0(A)\ra 0,$$
où là encore les $H^i$ sont calculés pour la t-structure de c\oe{}ur $\cA$.

\bigskip

On peut décrire précisément certains sous-objets.

\begin{prop}
Soit $\cA$ une catégorie abélienne, c\oe{}ur d'une t-structure bornée sur une catégorie triangulée $\cD$, et soit $(\cF, \cT)$ un couple de torsion sur $\cA$. On note $\cA^\sharp$ le basculement de $\cA$ correspondant.

Soit $A$ un objet de $\cT$. Alors les sous-objets de $A$ dans $\cA^{\sharp}$ sont exactement les objets $B$ de $\cT$ munis d'une flèche $B\ra A$ dont le noyau dans $\cA$ est un objet de $\cF$.
\end{prop}

\begin{proof}
Soit $f : B\ra A$ comme dans l'énoncé. Complétons cette flèche en un triangle distingué dans $\cD$
$$B\ra A\ra C\ra B[1],$$
où $C$ est le complexe $B\ra A,$ $B$ étant placé en degré $-1$. Alors $H^{-1}(C)=\Ker f$ est un objet de $\cF$ par hypothèse, et $H^0(C)=\mathrm{Coker} f$ est un objet de $\cT$ comme quotient dans $\cA$ d'un objet de $\cT$ --- voir la remarque \ref{remarque:stabilite-quotient}. Ainsi, $C$ est un objet de $\cA^{\sharp}$, d'où une suite exacte 
$$0\ra B\ra A\ra C\ra 0$$
dans $\cA^\sharp$, qui montre que $B$ est bien un sous-objet de $A$ dans $\cA^\sharp$.

L'énoncé réciproque a une preuve similaire.
\end{proof}

\begin{exem}\label{exemple:tilt-stab}
Soit $\cA$ une catégorie abélienne, c\oe{}ur d'une t-structure bornée sur une catégorie triangulée $\cD$, et soit $Z$ une fonction de stabilité sur $\cA$. Soit $\cP$ le découpage de $\cD$ correspondant. Alors, pour tout $\phi\in \mathopen]0,1]$, le couple $(\cP(\mathopen]0, \phi]), \cP(\mathopen]\phi, 1]))$ est un couple de torsion sur $\cA$. C'est le c\oe{}ur de la condition de stabilité obtenue par action de la rotation d'angle $\phi$ sur $(\cA, Z)$.
\end{exem}

\begin{exem}
Soit $\cD$ une catégorie triangulée, et soient $\cP$ et $\cQ$ deux découpages de~$\cD$. Notons $\cA$ (resp. $\cB$) le c\oe{}ur de $\cP$ (resp. $\cQ$), et supposons que tout objet de $\cA$ soit extension d'objets de $\cB$ et $\cB[1]$ dans $\cD$ --- autrement dit, supposons que $\cP(\mathopen]0,1])$ soit une sous-catégorie pleine de $\cQ(\mathopen]-1,1])$. Alors le couple
$$(\cA\cap\cQ(\mathopen]-1,0]), \cA\cap\cQ(\mathopen]0,1]))$$
est un couple de torsion sur $\cA$, dont le basculement est $\cB$.
\end{exem}

Les deux exemples précédents ont pour conséquence immédiate l'énoncé suivant, qui souligne l'importance de la notion de basculement dans l'étude des conditions de stabilité.

\begin{prop}
Soit $\cD$ une catégorie triangulée, soit $\varepsilon>0$, et soit $\sigma=(\cP, Z)$ une condition de stabilité sur $\cD$. Il existe un voisinage ouvert $U$ de $\sigma$ dans $\Stab(\cD)$ tel que pour tout $\tau$ dans $U$, le c\oe{}ur de $\tau$ est un basculement de $\cP(]-\varepsilon, 1-\varepsilon])$.
\end{prop}

\subsection{Surfaces}\label{subsection:surfaces}

Soit $X$ une surface projective complexe. La construction de conditions de stabilité sur les surfaces apparaît dans \cite{Bridgeland08} pour les surfaces $\mathrm{K}3$, et dans \cite{ArcaraBertram13}. La version du théorème \ref{theoreme:def-Bridgeland} que nous avons donnée ici, due à Bayer, permet de simplifier largement les preuves de ces résultats.

Dans ce qui suit, on note $\mathrm{rg}(A)$ pour le rang d'un faisceau cohérent $A$ sur $X$ au point générique.

\bigskip

Commençons par quelques rappels sur la stabilité pour les faisceaux sans torsion sur les surfaces --- voir par exemple \cite{HuybrechtsLehn10}. Soit $\omega$ un élément du cône ample de $\NS(X)\otimes\R$.

\begin{defi}\label{definition:stable-classique}
Soit $A$ un faisceau cohérent sur $X$. 
\begin{enumerate}
\item La pente de $A$ est l'élément $\mu_\omega(A)$ de $\R\cup\{+\infty\}$ défini par 
$$\mu_\omega(A)=\frac{\omega.c_1(A)}{\mathrm{rg}(A)}$$
si $\mathrm{rg}(A)>0$, et $+\infty$ sinon.
\item On dit que $A$ est $\mu$-semistable (resp. $\mu$-stable) si pour tout sous-faisceau strict non nul $B$ de $A$, on a $\mu_\omega(B)\leq\mu_\omega(A)$ (resp.  $\mu_\omega(B)<\mu_\omega(A)$).
\item Supposons que $A$ est pur\footnote{$A$ est pur de dimension $d$ si le support de tout sous-faisceau cohérent non nul de $A$ est de dimension~$d$.} de dimension $d\in\{0, 1, 2\}$. Notons $P_{\omega, A}$ le polynôme
$$n\mapsto \int_X e^{n\omega}\ch(A)\td_X,$$
et $p_{\omega, A}$ le polynôme unitaire de degré $d$ qui lui est proportionnel. On dit que $A$ est semistable (resp. stable) au sens de Gieseker si pour tout sous-faisceau strict $B$ de $A$ et tout $n$ suffisamment grand, on a 
$$p_{\omega, B}(n)\leq p_{\omega, A}(n).$$
\end{enumerate}
\end{defi}

On parlera souvent de stabilité pour désigner la stabilité au sens de Gieseker. La $\mu$-stabilité implique la stabilité, et la semistabilité implique la $\mu$-semistabilité.

\begin{rema}
Les définitions ci-dessus impliquent que les faisceaux de torsion sont tous $\mu$-semistables de pente $+\infty$, et que les faisceaux $\mu$-semistables de rang non nul sont sans torsion.
\end{rema}

\begin{rema}
Si $\omega$ est la classe d'équivalence numérique d'un fibré en droites $H$, le théorème de Hirzebruch-Riemann-Roch montre que l'on a 
$$P_{\omega, A}(n)=\chi(A\otimes H^{\otimes n})$$
pour tout entier $n$. Un faisceau semistable au sens de Gieseker sera toujours supposé pur.
\end{rema}

On utilisera l'existence de filtrations de Harder-Narasimhan pour la $\mu$-stabilité : tout objet $A$ de $\mathrm{Coh}(X)$ a une filtration finie
$$0=A_0\subset A_1\subset\cdots\subset A_n=A$$
telle que
\begin{itemize}
\item pour tout $i\geq 0$, $A_{i+1}/A_i$ est $\mu$-semistable;
\item pour tout $i\geq 0$, $\mu_\omega(A_{i+1}/A_i)>\mu_\omega(A_{i+2}/A_{i+1})$.
\end{itemize}
En particuler, $A_1$ est le sous-faisceau de torsion de $A$. Les $\mu_\omega(A_{i+1}/A_i)$ sont les \emph{pentes} de $A$, et l'on notera $\mu_\omega(A_1)=\mu_\omega^+(A)$, $\mu_\omega(A/A_{n-1})=\mu_\omega^-(A)$.

\bigskip

La $\mu$-stabilité ne permet pas de donner directement une condition de stabilité sur $\mathrm{Coh}(X)$ ou $D^b(X)$. En effet, si $A$ est un faisceau cohérent sur $X$ supporté en codimension~$2$, alors $\omega.c_1(A)=0$ et $\mathrm{rg}(A)=0$. Cependant, l'existence de filtrations de Harder-Narasimhan permet de construire des catégories abéliennes dans $D^b(A)$ par basculement comme dans l'exemple \ref{exemple:tilt-stab}.

Par ailleurs, la notion de stabilité au sens de Gieseker est plus adaptée pour traiter des faisceaux cohérents arbitraires, mais elle n'est pas formulée a priori en terme de pentes. Pour motiver l'introduction de la charge centrale adaptée, considérons un faisceau $A$ sans torsion sur $X$. 
Notons 
$$P_{\omega, A}(X)=a(X^2+bX+c)=\int_X e^{X\omega}\ch(A)\td_X,$$
et considérons
$$W_{n\omega}(A)=-i\int_X e^{in\omega}\ch(A)\td_X=-i\,P_{\omega, A}(in)=a(bn+i(n^2-c)).$$
Pour $n$ suffisamment grand, la partie imaginaire de $W_{n\omega}(A)$ est strictement positive, et la pente généralisée de $W_{n\omega}(A)$, au sens de \ref{subsection:stab-ab}, est 
$$\frac{bn}{n^2-c}.$$
Soit maintenant $B$ un sous-faisceau cohérent de $A$, avec 
$$P_{\omega, B}(X)=a'(X^2+b'X+c').$$
Supposons pour fixer les idées $b$ et $b'$ positifs. Alors, pour $n$ suffisamment grand, on a 
$$\frac{b'n}{n^2-c'}\leq \frac{bn}{n^2-c}$$
si et seulement si $b'<b$ ou $b=b'$ et $c'\leq c$. Autrement dit, c'est le cas si et seulement si 
$$P_{\omega, B}(n)\leq P_{\omega, A}(n)$$
pour tout $n$ suffisamment grand. On retrouve ainsi la stabilité au sens de Gieseker --- à la condition de positivité de $b$ et $b'$ près, que l'on assurera par un procédé de basculement, puisque $b$ est, à une constante multiplicative strictement positive près, la pente $\mu_\omega(A)$.

\bigskip

La discussion ci-dessus suggère de considérer des charges centrales de la forme 
$$A\mapsto -i\int_X e^{in\omega}\ch(A)\td_X.$$
Plus généralement, la stabilité au sens de Gieseker n'étant pas stable par tensorisation avec un fibré en droites, on est amené à tordre les données comme suit. Soit $\beta$ un élément arbitraire de $\NS(X)_\R$ --- qui n'est donc pas nécessairement la classe d'un fibré en droites. On définit, pour tout faisceau cohérent de rang strictement positif
$$\mu_{\omega, \beta}(A)=\frac{\omega.(c_1(A)-\mathrm{rg}(A)\beta)}{\mathrm{rg}(A)}=\frac{\omega.c_1(A)}{\mathrm{rg}(A)}-\omega.\beta$$
et l'on pose $\mu_{\omega, \beta}(A)=+\infty$ si le rang de $A$ est nul. Si $\beta$ est de la forme $c_1(L)$, alors bien sûr, $\mu_{\omega, \beta}(A)=\mu_{\omega}(A\otimes L^\vee)$. Tout comme $\mu_\omega$ dont elle se déduit par translation, la pente $\mu_{\omega, \beta}$ définit une notion de stabilité pour laquelle existent des filtrations de Harder-Narasimhan. On notera encore $\mu_{\omega, \beta}^\pm(A)$ pour les plus petites et plus grandes pentes correspondantes.

Considérons maintenant 
$$Z_{\omega, \beta} : K(X)\ra\C, A\mapsto-\int_X e^{-i\omega-\beta}\ch(A).$$
Là encore, si $\beta=c_1(L)$, on a $Z_{\omega, \beta}(A)=Z_{\omega, 0}(A\otimes L^\vee).$ C'est la charge centrale que nous allons considérer.

\begin{rema}
Les normalisations dans la définition de $Z_{\omega, \beta}$ varient dans la littérature suivant les auteurs et le contexte. Les signes devant $\int_X$, $\omega$ et $\beta$ peuvent varier, ce qui est bien entendu accessoire à l'action de $\widetilde{GL}_2^+(\R)$ et à des basculements près. 

Dans le cas où $X$ est une surface $\mathrm{K}3$, on introduit plutôt
$$Z_{\omega, \beta}(A)=-\int_X e^{-i\omega-\beta}\ch(A)\sqrt{\td_X},$$
pour des raisons liées à l'utilisation systématique du réseau de Mukai. 
\end{rema}

Dans ce qui suit, notons 
$$\ch^\beta(A)=e^{-\beta}\ch(A), \ch_i^\beta(A)=(e^{-\beta}\ch(A))_i,$$
où l'indice $i\in\{0,1,2\}$ désigne la composante de codimension $i$. On a alors
\begin{align}
\ch_0^\beta(A) & =\mathrm{rg}(A), \notag \\
\ch_1^\beta(A) & =c_1(A)-\mathrm{rg}(A)\beta, \notag \\
\ch_2^\beta(A) & =\ch_2(A)-\beta.c_1(A)+\frac{1}{2}\beta^2\mathrm{rg}(A),\notag \\
\label{equation:Z-formule}
Z_{\omega, \beta}(A) & =-\int_X e^{-i\omega}\ch^B(A)=\Big(\frac{1}{2}\mathrm{rg}(A)\omega^2-\ch_2^\beta\Big)+i\,\mathrm{rg}(A)\mu_{\omega, \beta}(A).
\end{align}

Si $Z_{\omega, \beta}$ est la charge centrale d'une condition de stabilité sur $D^b(X)$, son c\oe{}ur doit être une catégorie abélienne dont les objets sont des objets $A$ de $D^b(X)$ tels que $\mathrm{rg}(A)\mu_{\omega, \beta}(A)\geq 0$. 

Soient donc $\cF_{\omega, \beta}$ et $\cT_{\omega, \beta}$ les sous-catégories pleines de $\mathrm{Coh}(X)$ dont les objets non nuls sont les faisceaux cohérents $A$ tels que $\mu_{\omega, \beta}^+(A)\leq 0$ et $\mu_{\omega, \beta}^-(A)> 0$ respectivement. Les propriétés de base de la $\mu$-stabilité garantissent que $(\cF_{\omega, \beta}, \cT_{\omega, \beta})$ est un couple de torsion. 

\begin{defi}
Avec les notations précédentes, on note 
$$\mathrm{Coh}^{\omega, \beta}(X)$$
la catégorie abélienne obtenue comme basculement de $\mathrm{Coh}(X)$  à partir de $(\cF_{\omega, \beta}, \cT_{\omega, \beta})$. C'est le c\oe{}ur basculé de $\mathrm{Coh}(X)$.
\end{defi}

Voici le théorème principal d'existence \cite{Bridgeland08, ArcaraBertram13}.

\begin{theo}\label{theoreme:existence-stab}
Soit $A$ le cône ample de $X$ dans $\NS(X)_\R$. Pour tout $\omega\in A, \beta\in \NS(X)_\R$, 
$\sigma_{\omega, \beta}=(\mathrm{Coh}^{\omega, \beta}(X), Z_{\omega, \beta})$ est une condition de stabilité sur $D^b(X)$. De plus, la flèche 
$$\widetilde{GL}_2^+(\R)\times A\times \NS(X)_\R\ra \Stab(X), (g, \omega, \beta)\mapsto \sigma_{\omega, \beta}.g$$
est une immersion ouverte.
\end{theo}

Donnons quelques éléments de preuve --- nous renvoyons à \cite{MacriSchmidt17} pour détails et références et donnerons par ailleurs des arguments plus précis dans le cas où $X$ est une surface $\mathrm{K}3$. 

L'énoncé difficile est celui qui affirme que $\sigma_{\omega, \beta}$ est une condition de stabilité pour tout couple $(\omega, \beta)$ --- la continuité suit d'arguments de rigidité, et la propriété d'immersion ouverte est facile à vérifier car le théorème de déformation de Bridgeland nous permet de le faire après projection sur $\Hom(K(X)_{\mathrm{num}}, \C).$

La stratégie est simple : on commence par traiter le cas où $\omega$ et $\beta$ sont des classes de fibrés en droites, et l'on montre que $\sigma_{\omega, \beta}$ est une condition de stabilité vérifiant la condition de support par rapport à une forme quadratique explicite. Ensuite, le théorème \ref{theoreme:def-Bridgeland}, dans la version effective sous laquelle nous l'avons formulé, permet par prolongement de montrer que $\sigma_{\omega', \beta'}$ est une condition de stabilité quand $\omega'$ et $\beta'$ appartiennent à un voisinage contrôlé de $(\omega, \beta)$. Dans tout cela, l'ingrédient clé est le suivant, voir par exemple \cite{HuybrechtsLehn10}.

\begin{theo}[Inégalité de Bogomolov]\label{theoreme:bogomolov}
Soit $A$ un faisceau sans torsion, \mbox{$\mu$-semistable} par rapport à un certain fibré en droites ample. Soit $\beta\in \NS(X)_\R$. Alors 
$$\ch_1^\beta(A)^2\geq 2\mathrm{rg}(A)\ch_2^\beta(A).$$
\end{theo}

C'est l'inégalité quadratique ci-dessus qui, à quelques manipulations près pour faire intervenir $\omega$, fournit la condition de support.

Montrons tout d'abord que $Z=Z_{\omega, \beta}$ est une fonction de stabilité sur $\mathrm{Coh}^{\omega, \beta}(X)$. Un objet $E$ de $\mathrm{Coh}^{\omega, \beta}(X)$ s'insère dans une suite exacte
$$0\ra A[1]\ra E\ra B\ra 0,$$
où $A$ et $B$ sont des objets de $\cF_{\omega, \beta}$ et $\cT_{\omega, \beta}$ respectivement. Notant $Z(A[1])=-Z(A)$, la construction de $\cF_{\omega, \beta}$ et $\cT_{\omega, \beta}$, considérant (\ref{equation:Z-formule}), assure que les parties imaginaires de $Z(A[1])$ et $Z(B)$ sont positives --- il en va donc de même de celle de $Z(E)$. Par ailleurs, si le support de $B$ est de dimension strictement positive, la partie imaginaire de $Z(B)$ est non nulle elle aussi.

Supposons que $Z(E)$ soit réelle. Alors $B$ est de longueur finie, $Z(B)$ est strictement négatif si $B$ est non nul par (\ref{equation:Z-formule}),  et 
$$Z(E)=Z(B)-Z(A)=Z(B)+\ch_2^\beta(A)-\frac{1}{2}\mathrm{rg}(A)\omega^2,$$
tandis que 
$$\ch_1^\beta(A).\omega=0.$$
Le théorème de l'indice de Hodge montre que l'on a $(\ch_1^\beta(A))^2\leq 0$, et l'inégalité de Bogomolov implique 
$$\ch_2^\beta(A)\leq 0,$$
ce qui montre bien que $Z(E)$ est strictement négatif sauf si $B=0$, $\mathrm{rg}(A)=0$ et $\ch_2^\beta=0$. Dans ce cas, $A$ est de torsion, mais $\ch_1^\beta(A)$ est alors effectif, ce qui contredit l'égalité $\ch_1^\beta(A).\omega=0$ si $\ch_1^\beta(A)$ est non nul. Nous avons bien montré que $Z$ est une fonction de stabilité sur $\mathrm{Coh}^{\omega, \beta}(X)$.

\bigskip

Il faut ensuite montrer l'existence de filtrations de Harder-Narasimhan. Dans le cas qui nous occupe, la discrétude de l'image de $Z$ nous permet --- par le même argument que pour la stabilité usuelle --- de montrer seulement que $\mathrm{Coh}^{\omega, \beta}(X)$ est une catégorie noethérienne, voir \cite[Section 7]{Bridgeland08}. 

\bigskip

Enfin, il faut montrer la condition de support par rapport à une forme quadratique explicite. L'idée est la suivante : soit $n$ un entier suffisamment grand. Si $E$ est un objet $\sigma_{n\omega, \beta}$-semistable, le choix de $Z$, comme nous l'avons expliqué au début de cette section, permet de montrer que $E$ est --- à décalage près --- un fibré semistable au sens de Gieseker --- voir par exemple \cite[Section 14]{Bridgeland08}. L'inégalité de Bogomolov fournit une inégalité pour $E$ telle que demandée par la condition de support. Il s'agit ensuite de montrer que cette inégalité vaut pour les objets $\sigma_{\omega, \beta}$-semistables. Partant d'un tel objet~$E$, on applique la méthode de wall-crossing pour les formes quadratiques utilisée dans la preuve du lemme \ref{lemme:calcul-cle}.

On considère le wall-crossing pour $E$ le long du chemin $t\mapsto \sigma_{t\omega, \beta}$. On montre alors qu'il existe un entier $n$ tel que les facteurs de Harder-Narasimhan de $E$ pour $\sigma_{n\omega, \beta}$ sont semistables au sens de Gieseker. Cela permet de conclure à la validité d'une inégalité à la Bogomolov pour $E$ à partir de celle de ses facteurs pour $\sigma_{n\omega, \beta}$.

\subsection{Surfaces $\mathrm{K}3$}

Soit maintenant $X$ une surface $\mathrm{K}3$ --- rappelons qu'il s'agit d'une surface projective lisse, irréductible, simplement connexe, de fibré canonique trivial. On renvoie à \cite{HuybrechtsK3} pour les propriétés générales de ces surfaces. Suivant \cite{Bridgeland08}, décrivons plus précisément une composante connexe de l'espace des conditions de stabilité.

Soit 
$$\Lambda=\Z\oplus \NS(X)\oplus\Z$$
le \emph{réseau de Mukai} de $X$, que l'on munit de la forme quadratique
$$(r, c, s)(r', c', s')=c.c'-rs'-r's.$$
Le réseau de Mukai est pair.

Si $A$ est un élément de $K(X)$, on note $v(A)=\ch(A)\sqrt{\td_X}\in\Lambda$ le \emph{vecteur de Mukai} de $A$. Dans le cas présent, le théorème de Riemann-Roch montre 
$$\chi(A, B)=-(v(A), v(B))$$
pour tous $A, B$ dans $K(X)$.

Étant donnée une charge centrale $Z : K(X)\ra\C$ supposée se factoriser par l'équivalence numérique, nous noterons encore $Z : \Lambda\ra\C$ le morphisme tel que pour tout $A\in K(X)$, 
$$Z(A)=Z(v(A)).$$

\begin{prop}\label{proposition:stabilite-intersection}
Soit $\sigma$ une condition de stabilité sur $X$, et soit $A$ un objet $\sigma$-stable de $D^b(X)$. Alors 
$$(v(A), v(A))\geq -2.$$
\end{prop}

\begin{proof}
Il faut montrer l'inégalité $\chi(A, A)\leq 2$, i.e. 
$$\sum_{i\in\Z} (-1)^i\dim\Hom(A, A[i])\leq 2.$$
Soit $\phi$ la phase de $A$. Pour tout entier $i$, $A[i]$ est stable de phase $\phi+i$. En particulier, pour $i<0$, on a $\Hom(A, A[i])=0$.

Par dualité de Serre, on a 
$$\dim\Hom(A, A[i])=\dim\Hom(A[i], A[2])=\dim\Hom(A, A[2-i]),$$
donc $\Hom(A, A[i])=0$ pour tout $i>2$, et $\Hom(A, A[2])=\Hom(A, A),$ qui est de dimension $1$ car $A$ est stable --- considérer image et noyau d'un endomorphisme de $A$.

Finalement, on a $\chi(A, A)=2-\dim\Hom(A, A[1])\leq 2.$
\end{proof}

Comme on l'a mentionné plus haut, les normalisations liées aux conditions de stabilité que l'on vient de construire sont différentes dans le cas des surfaces $\mathrm{K}3$, ce afin de travailler de manière compatible à la structure quadratique du réseau de Mukai. Nous écrirons donc, avec les notations de \ref{subsection:surfaces}, 
$$Z_{\omega, \beta}(A)=-\int_X e^{-i\omega-\beta}\ch(A)\sqrt{\td_X}.$$
La partie imaginaire de $Z_{\omega, \beta}$ étant inchangée par rapport à la convention de \ref{subsection:surfaces}, on ne change pas la notation $\mathrm{Coh}^{\omega, \beta}(X)$. Le théorème \ref{theoreme:existence-stab} devient dans ce contexte l'énoncé suivant --- remarquer la condition supplémentaire sur $\omega$, qui vient du changement de normalisation.

\begin{theo}[\cite{Bridgeland08}]
Soit $A$ le cône ample de $X$ dans $\NS(X)_\R$. Pour tout $\omega\in A, \beta\in \NS(X)_\R$ tel que $\omega^2>2$, $\sigma_{\omega, \beta}=(\mathrm{Coh}^{\omega, \beta}(X), Z_{\omega, \beta})$ est une condition de stabilité sur $D^b(X)$. De plus, la flèche induite
$$\widetilde{GL}_2^+(\R)\times \{\omega\in A, \omega^2>2\}\times \NS(X)_\R\ra \Stab(X), (g, \omega, \beta)\mapsto \sigma_{\omega, \beta}.g$$
est une immersion ouverte.
\end{theo}

\bigskip

Soit $\Stab^0(X)$ une composante connexe de $\Stab(X)$ contenant les conditions de stabilité du théorème ci-dessus. 

Soit $\cP_0(X)$ le sous-ensemble de $\Hom(\Lambda, \C)$ constitué des $Z : \Lambda\ra\C$ tels que $\Ker Z$ est un sous-espace défini négatif de $\Lambda_\R$, et ne contient pas de racine de $\Lambda$, i.e. de $\delta\in\Lambda$ tel que $(\delta, \delta)=-2.$ On peut montrer que $\cP_0(X)$ est ouvert dans $\Hom(\Lambda, \C)$, \cite[Proposition 8.1 et Proposition 8.3]{Bridgeland08}.

\begin{theo}[\cite{Bridgeland08}]\label{theorem:stab-K3}
Il existe une composante connexe $\cP_0^+(X)$ de $\cP_0(X)$ telle que l'application naturelle 
$$q : \Stab^0(X)\ra \Hom(\Lambda, \C)$$
est un revêtement de $\cP_0^+(X)$.
\end{theo}

\begin{rema}\label{remarque:dual}
La forme de Mukai identifie $\Hom(\Lambda, \C)$ à $\Lambda_\C$. Sous cette identification, $\cP_0(X)$ devient l'ensemble des éléments de $\Lambda_\C$ dont les parties réelles et imaginaires engendrent un $2$-plan positif dans $\Lambda_\R$, privé des $\delta^\perp$, où $\delta$ parcourt les racines de $\Lambda$.
\end{rema}

\begin{rema}
En général, si $X$ est une surface arbitraire et $\Stab^0(X)$ une composante de $\Stab(X)$, il est faux que l'application $q : \Stab^0(X)\ra \Hom(\Lambda, \C)$ soit un revêtement de son image, voir par exemple \cite{Meinhardt12}.
\end{rema}

\begin{proof}
On commence par montrer que $q^{-1}(\cP_0^+(X))\ra\cP_0^+(X)$ est un revêtement, en suivant les idées de \cite{Bayer16}. L'idée est de remarquer que, pour $\sigma\in q^{-1}(\cP_0^+(X)),$ la proposition \ref{proposition:stabilite-intersection} n'est pas loin de fournir une forme quadratique qui prouve la condition de support pour $\sigma$.

Soit $Z\in \cP_0^+(X)$. Notons $p$ la projection orthogonale $\Lambda_\R\mapsto \Ker Z$, et écrivons
$$(v,v)=\Vert Z(v)\Vert ^2-\vvvert p(v)\vvvert ^2,$$
où $\Vert .\Vert $ est une forme quadratique définie positive sur $\R^2\simeq\C$, et $\vvvert .\vvvert $ la norme euclidienne induite par $-(.,.)$ sur $\Ker Z$.

Soit $\Delta$ l'ensemble des racines de $\Lambda$, et soit $\delta\in\Delta$. Par hypothèse, $Z(\delta)\neq 0$, et la formule ci-dessus montre
$$\vvvert p(\delta)\vvvert ^2=2+\Vert Z(\delta)\Vert ^2.$$
En particulier, pour toute constante $K$, il n'existe qu'un nombre fini de $\delta\in \Delta$ tels que $\Vert Z(\delta)\Vert \leq K$. On peut donc trouver une constante $C>0$ telle que pour tout $\delta\in\Delta$, $\Vert Z(\delta)\Vert \geq C.$

Considérons la forme quadratique 
$$Q_Z(v)=(v,v)+\frac{2}{C^2}\Vert Z(v)\Vert ^2=\alpha^2\Vert Z(v)\Vert ^2-\vvvert p(v)\vvvert ^2,$$
avec $\alpha^2=\frac{C^2+2}{C^2}$ et $\alpha>0$. Alors $Q_Z$ est définie négative sur $\Ker Z$. 

\bigskip

Soit $\sigma$ une condition de stabilité. La proposition \ref{proposition:stabilite-intersection} et la définition de $C$ montrent que si $A$ est un objet $\sigma$-stable de $D^b(X)$, alors $Q_Z(v(A))\geq 0$. En particulier, la proposition \ref{proposition:support-stable} montre que $\sigma$ satisfait la condition de support par rapport à $Q_Z$ si et seulement si $q(\sigma)$ appartient à l'ouvert $\cP(Z)$ des éléments $W$ de $\Hom(\Lambda, \C)$ tels que $\Ker W$ est un sous-espace défini négatif de $Q_Z$. Si un tel $\sigma$ existe, le théorème \ref{theoreme:def-Bridgeland} montre que $q^{-1}(\cP(Z))\ra\cP(Z)$ est un revêtement. En particulier, $Z$ est dans l'image de $q$.

\bigskip

Le théorème \ref{theoreme:def-Bridgeland} montre que l'image de $\Stab^0(X)$ par $q$ est ouverte. Soit $U$ l'intersection de cette image avec $\cP_0^+(X)$, et soit $S$ la préimage de $U$ dans $\Stab^0(X)$. Si $U\neq\cP_0^+(X)$, on peut trouver un élément $Z$ dans $\cP_0^+(X)$ adhérent à $U$ mais pas dans~$U$. L'argument précédent montre que $Z$ est dans l'image de $q$, contradiction. On a donc $U=\cP_0^+(X)$.

L'argument précédent montre que $q^{-1}(\cP_0^+(X))\ra\cP_0^+(X)$ est un revêtement localement au-dessus de $\cP_0^+(X)$. C'est donc un revêtement. 

\bigskip

Pour conclure, il reste seulement à montrer que l'image de $q$ est contenue dans $\cP_0^+(X)$. Il faut un type d'argument différent : comme dans le cas des courbes et de la preuve de la proposition \ref{proposition:stab-courbe}, on a besoin d'un résultat d'existence d'objets stables. On va donc utiliser l'énoncé profond suivant, qu'on ne démontrera pas.

\begin{prop}\label{proposition:toda-non-vide}
Soit $\sigma$ un élément de $\Stab^0(X)$, et soit $v$ un élément du réseau de Mukai $\Lambda$. Si $v^2\geq -2$, il existe un objet $\sigma$-semistable $A$ de $D^b(X)$ tel que $v(A)=v$.
\end{prop}

L'énoncé correspondant pour la stabilité à la Gieseker est dû à Kuleshov \cite{Kuleshov90}, O'Grady \cite{Ogrady97} et Yoshioka \cite{Yoshioka01}. Dans le cas qui nous occupe, il est dû à Toda \cite{Toda08}, voir \cite[Theorem 6.8]{BayerMacri14proj}. L'idée est de se ramener au cas de la stabilité à la Gieseker en considérant des conditions de stabilité $\sigma_{n\omega, \beta}$ avec $n\ra \infty$, puis d'étudier en détail le wall-crossing pour passer de $\sigma_{n\omega, \beta}$ à une condition de stabilité arbitraire dans la composante connexe $\Stab^0(X)$. Ce type d'argument utilise les espaces de modules d'objets semistables dont nous discuterons les propriétés dans la section suivante. Il s'agit de montrer que certains invariants motiviques des espaces de modules sont préservés par wall-crossing.

Soit maintenant $\sigma$ un élément de $\Stab^0(X)$ de charge centrale $Z$. Si $\delta$ est une racine de $\Lambda$, la proposition montre que $\delta$ est de la forme $v(A)$ pour $A$ $\sigma$-semistable. Cela implique $Z(v(A))\neq 0$ par condition de support. Par ailleurs, supposons par l'absurde que $\Ker Z$ ne soit pas un sous-espace défini négatif de $\Lambda_\R$. Quitte à remplacer $Z$ par une petite déformation, on peut supposer qu'il existe un élément $v$ de $\Lambda$ tel que $v^2\geq 0$ et $Z(v)=0$, ce qui est là encore une contradiction.
\end{proof}

La géométrie du revêtement 
$$\Stab^0(X)\ra \cP_0^+(X)$$
est mal comprise. Dans \cite{Bridgeland08}, Bridgeland montre que le groupe des transformations de ce revêtement s'identifie au groupe des automorphismes de $D^b(X)$ qui laissent la composante $\Stab^0(X)$ globalement invariante et agissent trivialement sur $\Lambda$. Il conjecture que $\Stab^0(X)$ est simplement connexe, et globalement invariant par tous les automorphismes de $D^b(X)$. Cette conjecture est démontrée dans \cite{BayerBridgeland17} quand $X$ a nombre de Picard $1$.

\section{Espaces de modules}

\subsection{Généralités}

Nous décrivons brièvement les problèmes de modules associés aux conditions de stabilité à la Bridgeland. Par rapport à la construction des espaces de modules de faisceaux stables au sens de Gieseker, deux difficultés supplémentaires apparaissent : il semble difficile en général de borner le problème de modules et, même en supposant une telle finitude, l'espace de modules n'est pas a priori obtenu par une construction de théorie géométrique des invariants. En particulier, rien ne garantit qu'il soit projectif. 

Les techniques en jeu dans les énoncés qui suivent étant assez différentes de celles discutées dans ce texte, nous ne donnons pas de preuve.

\bigskip

Soit $X$ une variété projective lisse sur $\C$. Si $S$ est un schéma --- ou un espace algébrique --- localement de type fini sur $\C$, rappelons qu'un \emph{complexe $S$-parfait} est un complexe (pas nécessairement borné) de faisceaux quasi-cohérents sur $X\times S$ qui est, localement sur $S$, quasi-isomorphe à un complexe borné de faisceaux plats de présentation finie. On dira aussi qu'il s'agit d'une famille plate d'objets de $D^b(X)$.

Soit $\mathfrak M$ le $2$-foncteur en groupoïdes qui envoie un schéma $S$ localement de type fini sur $\C$ sur le groupoïde des complexes $S$-parfaits $E$ tels que pour tout $i<0$, et tout point complexe $s$ de $S$, on ait
$$\Hom(E_s, E_s[i])=0.$$

D'après \cite{Lieblich06}, $\mathfrak M$ est un champ d'Artin, localement de type fini sur $\C$.

\bigskip

Soit $\sigma$ une condition de stabilité (numérique) sur $D^b(X)$, et soit $v$ une classe d'équivalence numérique d'éléments de $K(X)$. Soit $Z$ la charge centrale associée à $\sigma$, et soit~$\phi$ un nombre réel tel que $Z(v)\in \R^*_+e^{i\phi}.$ On note $\mathfrak M_\sigma(v, \phi)$ le sous-champ de $\mathfrak M$ qui paramètre les objets $\sigma$-semistables dans $D^b(X)$, de classe $v$ et de phase $\phi$. On note~$\mathfrak M^s_\sigma(v, \phi)$
le sous-champ qui paramètre les objets stables.

\bigskip

Dans cette généralité, on sait très peu de choses sur les champs $\mathfrak M_\sigma(v, \phi)$ et $\mathfrak M^s_\sigma(v, \phi)$. Voici un exemple important où l'on peut contrôler ces objets.

Supposons que $X$ est de dimension $2$. Soit $\omega$ un élément du cône ample de $\NS(X)_\R$, et $\beta$ un élément de $\NS(X)_\R$. Soit $\sigma_{\omega, \beta}$ la condition de stabilité du théorème \ref{theoreme:existence-stab}, et notons $\mathfrak M_{\omega, \beta}(v, \phi)$ et $\mathfrak M^s_{\omega, \beta}(v, \phi)$ pour les champs correspondant à $\sigma_{\omega, \beta}$. 

\begin{theo}[\cite{Toda08, PiyaratneToda16}]\label{theorem:existence-modules}
Le champ $\mathfrak M_{\omega, \beta}(v, \phi)$ est un champ d'Artin de type fini sur $\C$, et $\mathfrak M^s_{\omega, \beta}(v, \phi)$ en est un ouvert.

Le champ $\mathfrak M^s_{\omega, \beta}(v, \phi)$ est une $\mathbb G_m$-gerbe au-dessus d'un espace algébrique $M^s_{\omega, \beta}(v, \phi)$, qui est propre si $\mathfrak M_{\omega, \beta}(v, \phi)=\mathfrak M^s_{\omega, \beta}(v, \phi)$.
\end{theo}

Les champs $\mathfrak M_{\omega, \beta}(v, \phi)$ et $\mathfrak M^s_{\omega, \beta}(v, \phi)$ ne dépendent pas, à isomorphisme canonique près, du choix de $\phi$ --- on les notera donc $\mathfrak M_{\omega, \beta}(v)$ et $\mathfrak M^s_{\omega, \beta}(v)$.

\begin{rema}
Si $X$ est une surface $\mathrm{K}3$, nous nous écarterons légèrement des notations ci-dessus et remplacerons $v$ par $v\sqrt{\td_X}.$
\end{rema}

Ce qui suit permet de relier espaces de modules au sens de la stabilité classique et au sens précédent. Par souci de simplicité, nous ne considérons que le cas où $\beta=0$, le cas général en étant une modification immédiate.

\begin{prop}\label{proposition:large-volume}
Avec les notations précédentes, supposons le rang de $v$ non nul, et $\mu_{\omega}(v)>0$. Alors, pour tout réel $t$ suffisamment grand, le champ $\mathfrak M_{t\omega}(v)$ admet pour espace de modules grossier l'espace de modules des faisceaux sans torsion semistables au sens de Gieseker.
\end{prop}

\begin{proof}
C'est essentiellement \cite[Proposition 14.2]{Bridgeland08}. Les arguments du début de \ref{subsection:surfaces} montrent directement que, sous les hypothèses de la proposition, les objets $\sigma_{t\omega}$-semistables de classe $v$ de $\mathrm{Coh}^{t\omega}(X)$ sont exactement les faisceaux sans torsion sur~$X$ qui sont semistables au sens de Gieseker, ce qui est ce que l'on voulait démontrer.
\end{proof}

\subsection{Le lemme de positivité}\label{subsection:positivite}

Soit $X$ une variété projective lisse sur $\C$. Soit $S$ un espace algébrique propre, irréductible sur $\C$ --- dans la discussion ci-dessous, nous supposerons $S$ lisse et projectif pour éviter l'introduction de complexes parfaits.

Soit $E$ un objet de $D^b(X\times S)$, dont on note $[E]$ la classe dans $K(X\times S)_{\mathrm{num}}$. Soient~$p$ et~$q$ les projections de $X\times S$ sur $X$ et $S$ respectivement. On définit, comme dans le cas où $E$ est un faisceau \cite{LePotier92}, \cite[8.1]{HuybrechtsLehn10}
$$\lambda_E : K(X)_{\mathrm{num}}\ra \NS(S)$$
par la composition 
\[
\xymatrix{
K(X)_{\mathrm{num}}\ar[r]^{p^*} & K(X\times S)_{\mathrm{num}}\ar[r]^{.[E]} & K(X\times S)_{\mathrm{num}}\ar[r]^{q_*} &K(S)_{\mathrm{num}}\ar[r]^{\det} & \NS(S).
}
\]

Soit $v$ la classe de $E_s$ dans $K(X)_{\mathrm{num}}$, où $s$ est un point quelconque de $s$. Le calcul de \cite[Lemma 8.1.2, (iv)]{HuybrechtsLehn10} montre que si $F$ est un faisceau localement libre de rang $r$ sur $X$, on a 
\begin{equation}\label{equation:fonctor-lambda}
(\lambda_E)_{|v^\perp}=\frac{1}{r}(\lambda_{E\otimes p^*F})_{|v^\perp}.
\end{equation}

Dans le cas classique, la flèche $\lambda_E$ est utilisée pour produire des fibrés amples sur $S$. On peut faire de même dans le cas des familles d'objets stables --- c'est le contenu de \cite[Theorem 1.1]{BayerMacri14proj}. 

Donnons-nous donc une condition de stabilité $\sigma$ sur $D^b(X)$, $v$ un élément de $K(X)_{\mathrm{num}}$, et $E\in \mathrm{Ob}(D^b(X\times S))$ une famille plate d'objets $\sigma$-semistables de classe $v$ dans $D^b(X)$. 

Soit $Z$ la charge centrale de $\sigma$. Soit $\Omega\in K(X)_{num, \R}$ tel que pour tout $w\in K(X)_{\mathrm{num}}$, on ait
$$\mathrm{Im} \Big(-\frac{Z(w)}{Z(v)}\Big)=\chi(\Omega, w),$$
où $\chi$ est la caractéristique d'Euler. En particulier, $\Omega\in v^\perp$. 

On note 
$$\ell_{\sigma, E}=\lambda_E(\Omega)\in \NS(S)_\R.$$
Des arguments standard utilisant (\ref{equation:fonctor-lambda}) montrent que cette construction définit une classe d'équivalence numérique de fibrés en droites $\ell_\sigma$ sur le champ $\mathfrak M_\sigma(v)$, et sur l'espace de module grossier $M_\sigma(v)$ s'il existe.

La construction de $\ell_{\sigma, E}$ montre immédiatement :

\begin{prop}
Si $\sigma$ est rationnelle, alors $\ell_{\sigma, E}$ est rationnel.
\end{prop}

\bigskip

Le théorème fondamental est le suivant.

\begin{theo}[Lemme de positivité]\label{theoreme:lemme-pos}
La classe $\ell_{\sigma, E}$ est nef. De plus, si $C$ est une courbe dans $S$, alors $\ell_{\sigma, E}.C=0$ si et seulement si pour tous points complexes $s$ et $s'$ généraux de $S$, les objets $\sigma$-semistables $E_{s}$ et $E_{s'}$ ont les mêmes facteurs de Jordan-Hölder à permutation près. 
\end{theo}

\begin{proof}
Donnons une idée de l'argument. Quitte à faire agir le groupe $\widetilde{GL}_2^+(\R)$, on peut supposer $Z(v)=-1$, de sorte que la classe $\Omega$ ci-dessus soit caractérisée par 
$$\chi(\Omega, w)=\mathrm{Im} Z(w)$$
pour tout $w$. 

Étant donnée une courbe irréductible et réduite $C$ dans $S$, un calcul via la formule de projection donne
$$\ell_{\sigma, E}.C=\mathrm{Im}(Z(p_*E_{|X\times C})).$$ 
On peut donc supposer $S=C$.

\bigskip

Soit $\cA$ le c\oe{}ur de $\sigma$. Si $p_*E$ est un objet de $\cA$, alors bien sûr $\mathrm{Im}(Z(p_*E))\geq 0$ par définition d'une condition de stabilité. 

En général, $p_*E$ n'est pas un objet de $\cA$, mais un argument délicat dû à Abramovich-Polishchuk et Polishchuk \cite{AP06, Polishchuk07} permet de montrer que
$$p_*(E\otimes q^*\cO(ns))$$
est un objet de $\cA$, où $s$ est un point lisse de $S$ et $n$ est un entier suffisamment grand. Pour conclure, il suffit de montrer que pour tout $n$ positif, on a 
$$\mathrm{Im}(Z(p_*(E\otimes q^*\cO(ns))))=\mathrm{Im}(Z(p_*E)),$$
donc que pour tout $n$ positif,
$$\mathrm{Im}(Z(p_*(E\otimes q^*\cO(ns))))=\mathrm{Im}(Z(p_*(E\otimes q^*\cO((n+1)s)))).$$
Les deux termes diffèrent manifestement de la quantité
$$\mathrm{Im}(Z(p_*(E_{s})))=\mathrm{Im}(Z(E_s))=\mathrm{Im}(Z(v))=0$$
car $Z(v)=-1$. Cela montre le premier énoncé.

\bigskip

Montrons une des deux directions du second énoncé : supposons que $S$ soit une courbe irréductible et que l'on ait 
$$Z(p_*E))\in\R.$$
Soit $s$ un point lisse de $S$. L'argument précédent montre que tous les $Z(p_*(E\otimes q^*\cO(ns)))$ sont des nombres réels négatifs --- en particulier, les $p_*(E\otimes q^*\cO(ns))$ sont semistables. La suite exacte dans $\cA$, pour $n$ assez grand,
$$0\ra p_*(E\otimes q^*\cO(ns))\ra p_*(E\otimes q^*\cO((n+1)s))\ra E_s$$
montre que les facteurs de Jordan-Hölder de $E_s$ sont des facteurs de Jordan-Hölder de $p_*(E\otimes q^*\cO((n+1)s))$. En particulier, ils ne dépendent pas de $s$, ce qui conclut.
\end{proof}

\section{Géométrie birationnelle de certaines variétés symplectiques holomorphes}

Dans cette section, on décrit les travaux de Bayer-Macr\`i \cite{BayerMacri14proj, BayerMacri14}.

\subsection{Variétés symplectiques holomorphes : rappels}

On renvoie à \cite{Beauville83, Huybrechts99, Markman11} pour plus de détails sur les variétés symplectiques holomorphes et leur importance en géométrie algébrique et complexe, et à \cite{HassettTschinkel09} pour des éléments de leur géométrie birationnelle.

\begin{defi}
Une variété symplectique holomorphe est une variété $X$ projective lisse sur $\C$, simplement connexe, telle que l'espace $H^0(X, \Omega^2_{X/\C})$ est engendré par une $2$-forme partout non dégénérée.
\end{defi}

En particulier, une variété symplectique holomorphe est de dimension paire, et son fibré canonique est trivial. Il suit qu'une application birationnelle entre variétés symplectiques holomorphes induit un isomorphisme entre les complémentaires de fermés de codimension $2$. En particulier, une telle application induit un isomorphisme entre les groupes de Néron-Severi des deux variétés.

\begin{rema}
Une variété symplectique holomorphe de dimension $2$ est une surface~$\mathrm{K}3$.
\end{rema}

Si $X$ est une variété symplectique holomorphe de dimension $2n$, le groupe de cohomologie singulière (modulo torsion) $H^2(X, \Z)$ est muni d'une forme quadratique naturelle~$q$, la \emph{forme de Beauville-Bogomolov}. Elle est de signature $(3, b-3)$ et se caractérise par l'existence d'une constante $C>0$ telle que pour tout $\alpha\in H^2(X, \Z)$, on ait
$$\int_X\alpha^{2n}=C q(\alpha)^n.$$

\bigskip

Fixons une variété symplectique holomorphe $X$ de dimension $2n$. Le \emph{cône ample} de $X$, noté $\Amp(X)$, est le cône dans $\NS(X)_\R$ engendré par les classes de diviseurs amples dans $X$. Le \emph{cône ample birationnel} de $X$, noté $\BAmp(X)$, est la réunion des $f^*\Amp(X')$, où $f : X\dashrightarrow X'$ parcourt l'ensemble des applications birationnelles de $X$ vers une variété symplectique holomorphe $X'$. Le cône ample birationnel n'a pas de raison d'être un cône, mais son adhérence en est un.

Notons par ailleurs $\BIG(X)$ et $\Mov(X)$ les cônes dans $\NS(X)_\R$ engendrés par les classes de diviseurs big et mobiles --- i.e. dont le lieu base d'une puissance positive est de codimension au moins $2$ --- respectivement.

On vérifie immédiatement que le cône ample birationnel de $X$ est contenu dans l'intersection de $\BIG(X)$ et $\Mov(X)$. 

Soit par ailleurs $D$ un élément du cône big de $X$. Alors la paire $(X, D)$ est de log-type général, on peut donc appliquer le programme du modèle minimal à la paire $(X, \lambda D)$ pour $\lambda>0$ suffisamment petit --- \cite{BCHM10} et \cite{MatsushitaZhang13} pour ce cas précis --- et en déduire l'existence d'un morphisme birationnel entre variétés symplectiques holomorphes
$$f : X'\dashrightarrow X$$
tel que $f^*D$ est big et nef. On vérifie que si de plus $f^*D$ est ample, alors $X'$ est unique. On dira que $(X', f^*D)$ est un modèle minimal de $(X, D)$.

\bigskip

Soit $X$ une surface $\mathrm{K}3$, soit $\Lambda$ le réseau de Mukai de $X$, et soit $v$ un élément primitif de $\Lambda$. Soient $\omega$ un élément du cône ample de $X$, et $\beta$ un élément de $\NS(X)_\R$. On étend la définition \ref{definition:stable-classique} de semistabilité (au sens de Gieseker) en remplaçant 
$$P_{\omega, A} : n\mapsto \int_X e^{n\omega}\ch(A)\td_X$$
par 
$$P_{\omega, A}^\beta : n\mapsto \int_X e^{n\omega}\ch^\beta(A)\td_X.$$

On note $M_\omega^{\beta}(v)$ l'espace de modules grossier des faisceaux sur $X$ semistables correspondant.

Le théorème suivant suit de \cite{Kuleshov90, Ogrady97, Yoshioka01}. Disons que $v=(r, c, s)\in\Lambda=\Z\oplus \NS(X)\oplus\Z$ est \emph{strictement positif} si $v^2\geq -2$ et si l'une des conditions suivantes est vérifiée :
\begin{enumerate}
\item $r>0$;
\item $r=0$, $c$ est effectif et $s\neq 0$;
\item $r=0, c=0$ et $s>0.$
\end{enumerate}

\begin{theo}\label{theorem:basic-module-K3}
On reprend les notations précédentes : 
\begin{enumerate}
\item Si $v$ est de la forme $nv'$, avec $n>0$ et $v'$ strictement positif, alors $M_\omega^{\beta}(v)$ est non-vide et projectif.
\item Si de plus $v$ est primitif et $\omega$ est générique, alors $M_\omega^{\beta}(v)$ paramètre des objets stables, et c'est une variété symplectique holomorphe de dimension $v^2+2.$
\end{enumerate}

Supposons la seconde condition vérifiée, et $v^2> 0$. Alors il existe une isométrie canonique
$$\theta : v^\perp\ra \NS(M_\omega^\beta(v)),$$
où l'espace $\NS(M_\omega^\beta(v))$ est muni de la forme de Beauville-Bogomolov. 

Si $v^2=0$, il existe une isométrie canonique
$$\theta : v^\perp/\langle v\rangle\ra \NS(M_\omega^\beta(v)).$$
\end{theo}

\begin{rema}
Dans le théorème, nous ne considérons que le groupe de Néron-Severi $\NS(M_\omega^\beta(v))$ et non toute la cohomologie de degré $2$. La notation $v^\perp$ désigne donc l'orthogonal de $v$ dans $\Lambda$, la partie algébrique du réseau de Mukai usuel.
\end{rema}

\begin{rema}
Nous n'aurons pas besoin de préciser la condition de généricité sur~$\omega$ --- elle vaut sur un ouvert dense du cône ample de $X$.
\end{rema}

\begin{rema}\label{remarque:theta-lambda}
L'isométrie $\theta$ est induite à un facteur strictement positif près par la flèche $-\lambda_E$ de \ref{subsection:positivite}, où $E$ est une famille quasi-universelle sur $X\times M_\omega^\beta(v)$ au sens de \cite[4.6]{HuybrechtsLehn10}.
\end{rema}

\subsection{Projectivité et géométrie des espaces de modules}\label{subsection:proj}

Soit $X$ une surface $\mathrm{K}3$, soit $\Lambda$ le réseau de Mukai de $X$, et soit $v$ un élément de $\Lambda$. On dit qu'une condition de stabilité numérique $\sigma$ sur $X$ est \emph{générique} (par rapport à $v$) si $\sigma$ appartient à une chambre de la décomposition de $\Stab(X)$ associée à $v$ par la proposition~\ref{proposition:murs-chambres}. On note encore $\Stab^0(X)$ la composante connexe de $\Stab(X)$ contenant les $\sigma_{\omega, \beta}$.

\bigskip

On peut généraliser le théorème \ref{theorem:existence-modules} comme suit dans le cas des surfaces $\mathrm{K}3$, par le même genre d'arguments.

\begin{theo}[\cite{Toda08}]\label{theorem:existence-modules-K3}
Soit $\sigma\in \Stab^0(X)$. Le champ $\mathfrak M_{\sigma}(v)$ est un champ d'Artin de type fini sur $\C$, et $\mathfrak M^s_{\sigma}(v)$ en est un ouvert.

Le champ $\mathfrak M^s_{\sigma}(v)$ est une $\mathbb G_m$-gerbe au-dessus d'un espace algébrique $M^s_{\sigma}(v)$, qui est propre si $\mathfrak M_{\sigma}(v)=\mathfrak M^s_{\sigma}(v)$.
\end{theo}

La proposition \ref{proposition:toda-non-vide} garantit que les champs $\mathfrak M_\sigma(v)$ sont non-vides dès que $v^2\geq -2$. De plus, un argument de déformation montre le résultat suivant.

\begin{prop}\label{proposition:lisse-facile}
Supposons $v$ primitif et $v^2\geq -2.$ Si $\sigma$ est générique, alors $\mathfrak M_{\sigma}(v)=\mathfrak M^s_{\sigma}(v)$. L'espace de modules grossier $M_\sigma(v):=M^s_\sigma(v)$ est un espace algébrique symplectique holomorphe de dimension $v^2+2$.
\end{prop}

Si $\sigma$ est une condition de stabilité sur $X$, notons $\ell_\sigma$ la classe d'équivalence numérique de fibrés en droites sur $\mathfrak M_\sigma(v)$ (ou sur l'espace de modules grossier associé s'il existe) de \ref{subsection:positivite}. Le théorème \ref{theoreme:lemme-pos} montre que $\ell_\sigma$ est nef.

On peut calculer $\ell_\sigma$ comme suit, en suivant la construction de \ref{subsection:positivite}. Soit $E$ un objet quasi-universel sur $M_\sigma(v)$ --- qu'on obtient par une multisection de degré $d$ de la $\G_m$-gerbe $\mathfrak M_{\sigma}(v)\ra M_{\sigma}(v).$

Soit $\Omega_\sigma\in\Lambda_\R$ tel que $(\Omega_\sigma,w)=\mathrm{Im}\Big(\frac{Z(w)}{Z(v)}\Big)$ pour tout $w$ dans $\Lambda$. Alors
\begin{equation}\label{equation:ell}
\ell_\sigma=-\frac{1}{d}\lambda_E(\Omega_\sigma).
\end{equation}

\bigskip

\begin{theo}[\cite{BayerMacri14proj}, Theorem 1.3]
Supposons $\sigma$ générique. Alors l'espace de modules grossier $M_\sigma(v)$ est une variété projective.
\end{theo}

\begin{proof}
Donnons une indication, suivant Bayer-Macr\`i dans \cite{BayerMacri14proj}, qui s'inspirent du résultat partiel de \cite{MYY14}. On traite le cas $v^2>0$. Soit $Z$ la charge centrale de $\sigma$.

On commence par montrer par un argument explicite que l'on peut supposer, quitte à remplacer $\sigma$ par une condition de stabilité dans la même chambre pour $v$, qu'il existe un élément primitif $w$ de $\Lambda$ tel que $w^2=0$ tel que $Z(v)$ et $Z(w)$ soient alignés. 

Si $\sigma$ est générique par rapport à $w$, ce que nous allons supposer pour simplifier, alors la proposition \ref{proposition:lisse-facile} montre que $M_\sigma(w)$ est une surface $\mathrm{K}3$ $Y$, qui est automatiquement projective. La catégorie dérivée $D^b(X)$ est équivalente à une catégorie de la forme $D^b(Y, \alpha)$, où $\alpha$ est un élément du groupe de Brauer de $Y$. Fixons une telle équivalence.

La catégorie $D^b(Y, \alpha)$ est une variante de $D^b(Y)$ dans laquelle les faisceaux cohérents sont remplacés par les faisceaux cohérents tordus par $\alpha$, voir \cite{Lieblich07} pour cette notion. La théorie des conditions de stabilité pour les surfaces $\mathrm{K}3$ s'étend à $D^b(Y, \alpha)$ --- on parle de surface $\mathrm{K}3$ tordue --- et l'on peut montrer que la condition d'alignement de $Z(v)$ et $Z(w)$ permet d'identifier $M_\sigma(v)$, via l'équivalence ci-dessus, à un espace de modules d'objets semistables dans $D^b(Y, \alpha)$ pour une condition de stabilité de la forme $\sigma_{\omega, \beta}$, et que de plus la numérologie de $v$ et $w$ identifie cet espace de modules à un espace de modules de faisceaux (tordus) stables au sens de Gieseker pour $(\omega, \beta)$ --- qui est projectif par la variante tordue de la construction classique.
\end{proof}

La démonstration précédente montre plus : en identifiant $M_\sigma(v)$ à un espace de modules pour la stabilité au sens de Gieseker, elle permet d'étendre, en suivant l'identification, le résultat du théorème \ref{theorem:basic-module-K3} aux espaces de modules à la Bridgeland. Prenant en compte (\ref{equation:ell}) et la remarque \ref{remarque:theta-lambda}, on trouve : 

\begin{coro}
Si $v$ est primitif, et $\sigma$ générique, alors $M_\sigma(v)$ est une variété symplectique holomorphe de dimension $v^2+2$ si $v^2\geq -2$, vide sinon. Elle paramètre des objets stables.

Si $v^2> 0$, alors il existe une isométrie canonique
$$\theta_\sigma : v^\perp\ra \NS(M_\sigma(v)),$$
où l'espace $\NS(M_\sigma(v))$ est muni de la forme de Beauville-Bogomolov. 

Si $v^2=0$, il existe une isométrie canonique
$$\theta_\sigma : v^\perp/\langle v\rangle\ra \NS(M_\sigma(v)).$$
De plus, on a $\theta_\sigma(\Omega_\sigma)=\ell_\sigma.$
\end{coro}

Là encore, $\theta_\sigma$ est induit à un facteur près par $-\lambda_E$, avec les notations précédentes.

\begin{coro}
Si $v$ est primitif, et $\sigma$ générique, alors $\ell_\sigma$ est ample.
\end{coro}

\begin{proof}
On peut supposer $Z(v)=-1$, ce qui implique que $\Omega_\sigma$ est la partie imaginaire de l'élément $\alpha$ de $\Lambda_\C$ tel que $Z(w)=(\alpha, w)$ pour tout $w$.

La proposition \ref{proposition:rationnel-dense} et le caractère ouvert de l'amplitude permettent de supposer que $\sigma$ est rationnelle, donc que $\ell_\sigma$ est dans $\NS(M_\sigma(v))_\Q$. Le théorème \ref{theoreme:lemme-pos} montre que $\ell_\sigma$ a degré strictement positif sur toute courbe, et le corollaire précédent montre que, $q$ étant la forme de Beauville-Bogomolov, $q(\ell_\sigma)=\Omega_\sigma^2>0$, d'après le théorème \ref{theorem:stab-K3} et la remarque \ref{remarque:dual}. D'après \cite[6.3]{Huybrechts99}, cela implique l'amplitude de $\ell_\sigma$.
\end{proof}

\begin{rema}
L'amplitude de $\ell_\sigma$ est vraie même si $v$ n'est pas primitif.
\end{rema}

\bigskip

Supposons $v$ primitif, de carré au moins $2$. Soit $C$ une chambre de $\Stab^0(X)$ par rapport à $v$. La proposition \ref{proposition:murs-chambres} fournit une identification naturelle de tous les $M_\sigma(v)$, quand $\sigma$ varie dans $C$, à une même variété symplectique holomorphe $M$. Ce qui précède fournit une application 
$$\ell : C\ra \Amp(M), \sigma\mapsto \ell_\sigma,$$
dont on vérifie immédiatement qu'elle est analytique.

Les fibrés $\ell_\sigma$ permettent d'étudier les phénomènes de wall-crossing comme suit. Soit~$W$ un mur de $\Stab^0(X)$ par rapport à $v$, et soient $C_+$ et $C_-$ les deux chambres adjacentes à $W$. Soient $\sigma_\pm$ deux conditions de stabilité dans $C_+$ et $C_-$ respectivement, et soit $\sigma_0$ une condition de stabilité générique dans $W$. La fermeture de la semiamplitude montre que les familles quasi-universelles portées par chacun des espaces $M_{\sigma_\pm}(v)$ sont des familles d'objets $\sigma_0$-semistables, de sorte que l'on dispose de fibrés en droite $\ell_{\sigma_0, \pm}$ sur $M_{\sigma_\pm}(v)$. On a 
$$\theta_{\sigma_\pm}(\Omega_{\sigma_0})=\ell_{\sigma_0, \pm},$$
ce qui montre encore $q(\ell_{\sigma_0, \pm})>0$. Comme $\ell_{\sigma_0, \pm}$ est nef, cela implique que $\ell_{\sigma_0, \pm}$ est big par \cite[Corollary 3.10]{Huybrechts99}, puis que $\ell_{\sigma_0, \pm}$ est semiample par \cite[Theorem 3.3]{KM98}. On dispose donc de deux contractions birationelles
$$\pi_{\pm} : M_{\sigma_\pm}(v)\ra M_\pm$$
vers des variétés normales.

La géométrie des contractions $\pi_\pm$ est essentielle pour ce qui suit.

\subsection{Modèles birationnels et wall-crossing}

Nous pouvons enfin décrire les résultats de \cite{BayerMacri14}.

\begin{theo}\label{theoreme:main}
Soit $v$ un vecteur primitif de $\Lambda$ avec $v^2>0$. Soit $\sigma$ un élément générique de $\Stab_0(X)$. Alors il existe une application canonique
$$\ell : \Stab_0(X)\ra \NS(M_\sigma(v))_\R$$
vérifiant les propriétés suivantes :
\begin{enumerate}
\item $\ell$ est continue, analytique sur chaque chambre de $\Stab_0(X)$. Son image est l'intersection du cône big et du cône mobile de $M_\sigma(v)$;
\item Soit $C$ une chambre de $\Stab_0(X)$, et soit $\tau\in C$. Alors $(M_{\tau}(v), \ell_\tau)$ est un modèle minimal de $(M_{\sigma}(v), \ell(\tau))$. Il existe une application birationnelle canonique
$$f : M_{\sigma}(v)\dashrightarrow M_\tau(v)$$
telle que $\ell(\tau) = f^*\ell_\tau$, et 
$$\ell(C)=f^* \Amp(M_\tau(v)).$$
\end{enumerate}
\end{theo}

\begin{rema}
Dans le second énoncé, puisque $\ell_\tau$ est ample sur $M_{\tau}(v)$, $(M_{\tau}(v), \ell_\tau)$ est l'unique modèle minimal de $(M_{\sigma}(v), \ell(\tau))$.

Si $C$ est la chambre de $\Stab_0(X)$ contenant $\sigma$, alors $f$ est l'identification naturelle donnée par le programme du modèle minimal.
\end{rema}

\begin{rema}\label{rema:ell-et-murs}
Les conditions du théorème montrent que la flèche $\ell$ se factorise par la flèche
$$\Stab_0(X)\ra \cP_0^+(X)\subset \Lambda_\C\ra U\subset \Lambda_\R,$$
où $\Stab_0(X)\ra \cP_0^+(X)$ est la flèche qui à une condition de stabilité associe sa charge centrale, $\cP_0^+(X)\subset \Lambda_\C$ est l'inclusion de la remarque \ref{remarque:dual}, et $U$ est l'image de $\cP_0^+(X)$ par l'application partie imaginaire. Par analyticité, il suffit en effet de le montrer sur les conditions de stabilité génériques, auquel cas le résultat suit de la définition des $\ell_\sigma$.

En particulier, pour des raisons de dimension et d'analyticité de $\ell$, les murs de $\Stab_0(X)$ sont envoyés par $\ell$ sur des fermés d'intérieurs vides.
\end{rema}

\begin{rema}
Si $\sigma$ et $\tau$ ont même charge centrale, un théorème de Bridgeland déjà mentionné montre que $\sigma$ et $\tau$ diffèrent d'un automorphisme de $D^b(X)$, ce qui montre a priori que $M_\sigma(v)$ et $M_\tau(v)$ sont isomorphes.
\end{rema}

Voici un corollaire immédiat et important.

\begin{coro}\label{corollaire:wall-crossing-bir}
Soit $v$ un vecteur primitif de $\Lambda$ avec $v^2>0$, et soient~$\sigma$ et~$\tau$ deux éléments génériques de $\Stab_0(X)$. Alors $M_\sigma(v)$ et $M_\tau(v)$ sont birationnels. De plus, tout modèle birationnel symplectique holomorphe de $M_\sigma(v)$ est isomorphe à $M_\tau(v)$ pour un certain $\tau$ générique.
\end{coro}

\begin{proof}
Le premier énoncé est un cas particulier du théorème. Soit par ailleurs 
$$ f : M_\sigma(v)\dashrightarrow M$$
une application birationnelle, avec $M$ symplectique holomorphe. Soit $H$ un fibré ample sur $M$, et soit $\ell'=f^*H$. Alors $\ell'$ est big et mobile, donc on peut trouver une condition de stabilité $\tau\in \Stab_0(X)$ telle que $\ell'=\ell(\tau)$. Quitte à déformer $H$ dans le cône ample de $M$, la remarque \ref{rema:ell-et-murs} permet de supposer que $\tau$ est générique.

Par construction, $(M, H)$ est l'unique modèle minimal de $(M_\sigma(v), \ell(\tau)),$ donc le théorème précédent montre que $M$ est isomorphe à $M_\tau(v)$.
\end{proof}

\bigskip

Le théorème \ref{theoreme:main} donne une relation précise entre conditions de stabilité et wall-crossing d'une part, et programme du modèle minimal pour les espaces de modules d'autre part. Ce thème a été étudié par plusieurs auteurs, d'abord via les variations de quotient GIT comme dans \cite{Thaddeus94, DolgachevHu98}. Pour les surfaces à canonique trivial, les premiers exemples de wall-crossing sont dans \cite{ArcaraBertram13}, et le caractère birationnel du wall-crossing est prouvé pour les surfaces abéliennes dans \cite{MYY18}. Le cas du plan projectif est étudié notamment dans \cite{ABCH13, LiZhao16}, voir \cite{BHLRSWZ16} pour des résultats sur certaines surfaces plus générales, et \cite{Huizenga17} pour un survol de résultats récents.

\bigskip

La preuve du théorème \ref{theoreme:main} repose sur une analyse détaillée du wall-crossing : étant donné un mur $W$, soient $C_+$ et $C_-$ les deux chambres adjacentes à $W$. Soient $\sigma_\pm$ deux conditions de stabilité dans $C_+$ et $C_-$ respectivement, et soit $\sigma_0$ une condition de stabilité générique dans $W$. Il s'agit de montrer une version raffinée du corollaire \ref{corollaire:wall-crossing-bir} pour montrer la birationalité des espaces $M_{\sigma_\pm}(v)$ par une application birationnelle explicite, qui permette d'en déduire le théorème principal par recollement du comportement local au voisinage de chaque mur.

Il s'agit de comprendre comment le mur $W$ déstabilise certains objets stables pour $\sigma_+$ et $\sigma_-$. Par construction, notant $Z_0$ la charge centrale associée à $\sigma_0$, on peut trouver $w_0\in \Lambda$, qui n'est pas proportionnel à $v$, tel que $Z(v)$ et $Z(w_0)$ sont alignés dans $\mathbb H\cap \R_-$, et des objets~$D$, $\sigma_0$-stables de classe~$w_0$ dans $D^b(X)$, qui déstabilisent des objets~$E$, \mbox{$\sigma_+$-stables} de classe $v$, au sens où $Z(E)=v$, $Z(D)=w_0$, et $D$ est un sous-objet de~$E$ dans le c\oe{}ur de~$\sigma_0$.

Avec les notations précédentes, l'hypothèse de généricité sur $\sigma_0$ montre que le réseau primitif de rang $2$ contenant $v$ et $w_0$ ne dépend que de $W$. On le note $H_W$. Les propriétés des charges centrales données par le théorème \ref{theorem:stab-K3} montrent facilement que $H_W$ est hyperbolique. L'image par $Z_0$ de $W$ est contenue dans une demi-droite passant par l'origine.

La construction de $H_W$ relie les espaces de modules $M_{\sigma_\pm}(v)$ par la proposition suivante, essentiellement formelle étant donnée la finitude locale des murs.

\begin{prop}\label{prop:HN-reseau}
Supposons $\sigma_\pm$ suffisamment proches l'une de l'autre. Alors, si $E$ est un objet $\sigma_+$-stable de $D^b(X)$, les facteurs de Harder-Narasimhan de $E$ pour $\sigma_-$ ont pour vecteur de Mukai des éléments de $H_W$.
\end{prop}

Cette proposition permet d'étudier la relation entre $M_{\sigma_\pm}(v)$ via les propriétés du réseau $H_W$. Bayer-Macr\`i parviennent à déduire de cela une classification détaillée des murs dans $\Stab_0(X)$ que nous ne pouvons reproduire ici.

Le cas le plus difficile est celui où le mur $W$ est \emph{totalement semistable}, i.e. où il n'existe pas d'objets $\sigma_0$-stables. Dans ce cas, il est possible qu'aucun objet $\sigma_+$-semistable soit $\sigma_-$-semistable. Donnons un exemple des méthodes utilisées pour contrôler ce phénomène.

\bigskip

On reprend les notations précédentes. On va montrer comment les propriétés de $H_W$ permettent de contrôler la géométrie du wall-crossing. Nous ne traiterons qu'un cas particulier.

Rappelons que dans \ref{subsection:proj}, nous avons construit une contraction birationnelle 
$$\pi_+ : M_{\sigma_+}(v)\ra M_+$$
qui contracte exactement les courbes paramétrant les objets $\sigma_+$ stables dont les facteurs de Jordan-Hölder pour $\sigma_0$ sont génériquement constants.

Par ailleurs, supposons pour fixer les idées que $M_{\sigma_+}(v)$ est un espace de modules fin, et soit $\cE$ une famille universelle. On considère des filtrations de Harder-Narasimhan pour~$\sigma_-$. On peut montrer qu'elles existent en famille pour $\cE$. Soient $a_1, \ldots, a_m$ les vecteurs de Mukai des facteurs de Harder-Narasimhan de $\cE$ pour $\sigma_-$ en un point générique de $M_{\sigma_+}(v)$. Prenant les facteurs de la filtration de Harder-Narasimhan relative, on obtient une application rationnelle
$$\mathrm{HN} : M_{\sigma_+}(v)\dashrightarrow M_{\sigma_-}(a_1)\times\cdots\times M_{\sigma_-}(a_m).$$

La proposition \ref{prop:HN-reseau} montre que les $a_i$ sont des éléments de $H_W$. On a de plus $v=a_1+\cdots+a_m$, et $a_i^2\geq -2$ par la proposition \ref{proposition:stabilite-intersection}.

Soit $C$ une courbe dans $M_{\sigma_+}(v)$ contractée par $\mathrm{HN}$. Alors les facteurs de Harder-Narasimhan $E_i$ pour $\sigma_-$ des objets paramétrés par deux points génériques de $C$ sont les mêmes, ce qui implique que les facteurs de Jordan-Hölder de tels objets, pour $\sigma_0$, sont les mêmes, puisque ce sont des facteurs de Jordan-Hölder d'objets $E_i$, qui sont en nombre fini. On a donc montré :

\begin{prop}
Toute courbe contractée par $\mathrm{HN}$ est contractée par $\pi_+$. En particulier, 
$$v^2+2=\dim M_+\geq \dim (M_{\sigma_-}(a_1)\times\cdots\times M_{\sigma_-}(a_m))=\sum_{i=1}^m (a_i^2+2).$$
\end{prop}

L'égalité $m=1$ est équivalente au fait qu'un objet générique $\sigma_+$-stable de vecteur $v$ est $\sigma_-$-stable. En particulier, $M_{\sigma_+}(v)$ ont un ouvert en commun, donc sont birationnels, ce qui est essentiellement ce que l'on souhaitait montrer.

La suite de la démonstration consiste en l'exploitation de l'inégalité
$$v^2\geq 2(m-1)+\sum_{i=1}^m a_i^2.$$
Il s'agit de classifier précisément les cas qui permettent d'avoir $m=1$. Nous renvoyons à l'article \cite{BayerMacri14} pour une analyse détaillée. On montre que les cas $m>1$ correspondent toujours à l'apparition de certains vecteurs isotropes dans $H_W$ ou de racines de $\Lambda$.

La géométrie correspondant à l'existence de ces classes peut être décrite explicitement. Dans le cas de l'existence de classes isotropes, on montre que $\pi_+$ est essentiellement le morphisme de contractions de l'espace de modules de faisceaux stables au sens de Gieseker vers un espace de modules de faisceaux stables au sens de la pente.

Le second cas est plus difficile. On montre que les racines de $\Lambda$ apparaissant dans $H_W$ et contrôlant la géométrie du morphisme $\mathrm{HN}$ ci-dessus induisent des automorphismes de $D^b(X)$ --- les twists sphériques. Composant avec ces automorphismes, on peut se ramener au cas où $m=1$.

\subsection{Applications}

Les théorèmes précédents permettent de contrôler précisément les bords des cônes nef et mobiles des espaces $M_\sigma(v)$. Des arguments de densité permettent d'étendre ces résultats aux variétés hyperkähleriennes déformations de $M_\sigma(v)$ \cite{BHT15, Mongardi15}. On renvoie à \cite{HassettTschinkel16} pour un survol de ce cercle d'idées. Notons par ailleurs que des arguments de théorie ergodique permettent à Amerik et Verbitsky \cite{AmerikVerbitsky17} d'étudier des problèmes semblables pour des variétés symplectiques holomorphes arbitraires.

Donnons brièvement une application du corollaire \ref{corollaire:wall-crossing-bir}.

\begin{theo}[\cite{BayerMacri14}, Theorem 1.5]
Soient $X$ une surface $\mathrm{K}3$, $v$ un vecteur de Mukai primitif tel que $v^2=2n-2\geq 0$, et $\sigma$ une condition de stabilité générique sur $X$. Soit $L$ un fibré en droites non trivial sur $M_\sigma(v)$ tel que $q(L)=0$, où $q$ est la forme de Beauville-Bogomolov. Alors il existe une variété symplectique holomorphe $M$, birationnelle à $M_\sigma(v)$, qui admet une fibration $M\ra\mathbb P^n$ à fibres lagrangiennes.
\end{theo}

Markushevich \cite{Markushevich06} et Sawon \cite{Sawon07} avaient auparavant obtenu ce résultat pour certains schémas de Hilbert sur les surfaces $\mathrm{K}3.$ Un résultat similaire a été démontré dans \cite{Markman14}, et il est étendu dans \cite{Matsushita17} à toutes les déformations des $M_\sigma(v)$.

Un cas particulier du théorème est celui où $M_\sigma(v)$ est remplacé par un espace de modules de faisceaux stables au sens de Gieseker. De manière encore plus spéciale, on peut examiner le cas où $M_\sigma(v)$ est un espace de modules de faisceaux stables de rang zéro --- on vérifie alors immédiatement que l'hypothèse du théorème est automatiquement satisfaite.

Dans ce cas, étudié en détail dans \cite{Beauville91}, les points de $M_\sigma(v)$ paramètrent essentiellement la donnée d'une courbe $C$ dans $X$ dans un système linéaire donné, et d'un fibré en droites sur $X$ de degré fixé sur $C$. La projection sur le système linéaire fournit la fibration lagrangienne désirée. 

La preuve du théorème se ramène au cas précédent par transformées de Fourier-Mukai, puis déformation des conditions de stabilité pour déformer les espaces de modules en espaces de modules de faisceaux stables à la Gieseker. 

Plus précisément, on interprète la classe de $L$ dans $\NS(X)$ comme un vecteur de Mukai isotrope, qui fournit donc une surface $\mathrm{K}3$, $Y$, espace de modules d'objets sur~$X$. $M_\sigma(v)$ est alors un espace de modules d'objets sur $Y$, pour une certaine condition de stabilité, et un vecteur de Mukai de rang $0$. Déformant la condition de stabilité, et utilisant le corollaire \ref{corollaire:wall-crossing-bir} qui garantit que tous les espaces de modules en jeu sont birationnels les uns aux autres, on se ramène au cas de \cite{Beauville91}.

\newcommand{\etalchar}[1]{$^{#1}$}

\end{document}